\documentclass[12pt]{amsart}

\usepackage{accents}
\usepackage{appendix}
\usepackage{amsfonts}
\usepackage{amsmath}
\usepackage{amssymb}	
\usepackage{amsthm,bm}
\usepackage{array,booktabs,multirow}
\usepackage{braket}
\usepackage{centernot}
\usepackage{cite}
\usepackage{comment}
\usepackage{dsfont}
\usepackage{mathalfa}
\usepackage[shortlabels]{enumitem} 
\usepackage{etoolbox}
\usepackage{float}
\usepackage[hang, flushmargin]{footmisc}
\usepackage{latexsym}
\usepackage{lipsum}
\usepackage{needspace}
\usepackage{tikz}
\usepackage{hyperref}
\usetikzlibrary{matrix,arrows}  
\usepackage{kbordermatrix}

\renewcommand{\kbldelim}{(}
\renewcommand{\kbrdelim}{)}
\newcommand{\rvline}{\hspace*{-\arraycolsep}\vline\hspace*{-\arraycolsep}}

\theoremstyle{plain}
\newtheorem{thm}{Theorem}[section]
\newtheorem{cor}[thm]{Corollary}
\newtheorem{lem}[thm]{Lemma}
\newtheorem{prop}[thm]{Proposition}

\theoremstyle{definition}

\newtheorem{defn}[thm]{Definition}

\theoremstyle{remark}
\def\BI{\mathfrak{BI}}

\setlist[enumerate,1]{leftmargin=2em}

\def\H{\mathfrak H}

\def\F{\mathbb F}

\def\e{\varepsilon}

\title[Finite-dimensional modules of $\Re$ and $\H$]{Finite-dimensional modules of the universal Racah algebra and the universal additive DAHA of type $(C_1^\vee,C_1)$}

\author{Hau-Wen Huang}
\address{
Hau-Wen Huang\\
Department of Mathematics\\
National Central University\\
Chung-Li 32001 Taiwan
}
\email{hauwenh@math.ncu.edu.tw}

\begin{document}
\begin{abstract}
Assume that $\F$ is an algebraically closed field with characteristic zero. The universal Racah algebra $\Re$ is a unital associative $\F$-algebra defined by generators and relations. The generators are $A,B, C, D$ and the relations state that 
$$
[A,B]=[B,C]=[C,A]=2D
$$ 
and each of 
\begin{gather*}
[A,D]+AC-BA, \qquad [B,D]+BA-CB, \qquad [C,D]+CB-AC
\end{gather*}
is central in $\Re$. The universal additive DAHA (double affine Hecke algebra) $\mathfrak H$ of type $(C_1^\vee,C_1)$ is a unital associative $\F$-algebra generated by $t_0,t_1,t_0^\vee,t_1^\vee$ and the relations state that 
$$
t_0+t_1+t_0^\vee+t_1^\vee=-1
$$ 
and 
each of $t_0^2, t_1^2, t_0^{\vee 2}, t_1^{\vee 2}$ is central in $\mathfrak H$. Each $\H$-module is an $\Re$-module by pulling back via the algebra homomorphism $\Re\to \H$ given by 
\begin{eqnarray*}
A &\mapsto & \frac{(t_1^\vee+t_0^\vee)(t_1^\vee+t_0^\vee+2)}{4},
\\
B &\mapsto & \frac{(t_1+t_1^\vee)(t_1+t_1^\vee+2)}{4},
\\
C &\mapsto & \frac{(t_0^\vee+t_1)(t_0^\vee+t_1+2)}{4}.
\end{eqnarray*}
Let $V$ denote any finite-dimensional irreducible $\H$-module.
The set of $\Re$-submodules of $V$ forms a lattice under the inclusion partial order. 
We classify the lattices that arise by this construction. As a consequence, the $\Re$-module $V$ is completely reducible if and only if $t_0$ is diagonalizable on $V$.
\end{abstract}

\maketitle



\section{Introduction}\label{s:introduction}

Throughout this paper, we adopt the following conventions. Assume that $\F$ is an algebraically closed field with characteristic zero. An algebra is meant to be an associative algebra with unit $1$. Recall that the commutator $[\,,\,]$ and the anticommutator $\{\,,\,\}$ of two elements $X,Y$ in an algebra defined as follows:
\begin{align*}
[X,Y]&=XY-YX,
\\
\{X,Y\}&=XY+YX.
\end{align*}

Motivated by the coupling problem for three angular momenta,  L\'{e}vy-Leblond and L\'{e}vy-Nahas gave the first presentation for the Racah algebra \cite{Levy1965}. In \cite{zhedanov1988}, 
Granovski\v{i} and Zhedanov rediscovered the Racah algebra in an alternate presentation from the Racah problem for $\mathfrak{su}(2)$. The Racah algebra is now also explored in a broad range of areas including orthogonal polynomials,  distance regular graphs, superintegrable models and Leonard pairs \cite{LTRacah, LTRacah&DRG, Galbert, R&LD2014,gvz2014, R&BI2015,Kalnins1,Kalnins2, Kalnins3,Kalnins4, gvz2013, quadratic1991,quadratic1992, quadratic1989, hiddensymmetry1,hiddensymmetry3,Huang:R<BI,hidden_sym}. 
Given three parameters $\alpha,\beta,\gamma\in \F$ the Racah algebra has a presentation \cite[Section 1]{R&BI2015}
given by generators $A,B,C,D$ and the relations are
$$
[A,B]=[B,C]=[C,A]=2D
$$
and 
$$
\alpha=[A,D]+AC-BA,
\qquad
\beta=[B,D]+BA-CB,
\qquad 
\gamma=[C,D]+CB-AC.
$$
Inspired by \cite[Problem 12.1]{uaw2011}, we consider its central extension, denoted by $\Re$, obtained from the above presentation by reinterpreting the three parameters $\alpha,\beta,\gamma$ as central elements. We call $\Re$ the {\it universal Racah algebra} \cite{R&BI2015,SH:2019-1,SH:2017-1,Huang:R<BI}.

In \cite{Groenevelt2007} W. Groenevelt introduced an additive analogue of DAHA (double affine Hecke algebra) of type $(C_1^\vee,C_1)$ and used it to study generalized Fourier transforms. The additive DAHA of type $(C_1^\vee,C_1)$ also shows up in the context of Bannai--Ito polynomials \cite{BI&NW2016}. 
Given four parameters $k_0,k_1,k_0^\vee,k_1^\vee\in \F$ the algebra  has a presentation \cite[Proposition 2.12]{Groenevelt2007} given by generators  $t_0,t_1,t_0^\vee,t_1^\vee$ and relations
\begin{gather*}
t_0+t_1+t_0^\vee+t_1^\vee=-1,
\\
t_0^2=k_0,
\quad 
t_1^2=k_1,
\quad  
t_0^{\vee 2}=k_0^\vee, 
\quad 
t_1^{\vee 2}=k_1^\vee.
\end{gather*} 
In this paper we consider its central extension, denoted by $\H$,  obtained from the above presentation by reinterpreting the four parameters $k_0,k_1,k_0^\vee,k_1^\vee$ as central elements.
We call $\H$ the {\it universal additive DAHA of type $(C_1^\vee,C_1)$} \cite{Huang:BImodule,BI&NW2016}.

According to the results from \cite[Section 2]{R&BI2015} and \cite[Proposition 2]{BI&NW2016}, there exists a unique $\F$-algebra homomorphism $\zeta:\Re\to \H$ that sends
\begin{eqnarray*}
A &\mapsto & \frac{(t_1^\vee+t_0^\vee)(t_1^\vee+t_0^\vee+2)}{4},
\\
B &\mapsto & \frac{(t_1+t_1^\vee)(t_1+t_1^\vee+2)}{4},
\\
C &\mapsto & \frac{(t_0^\vee+t_1)(t_0^\vee+t_1+2)}{4}.
\end{eqnarray*}
The map $\zeta$ is shown to be injective \cite{Huang:R<BI} and the result can be considered as the Racah version of the algebra  monomorphism from the universal Askey--Wilson algebra into the universal DAHA of type $(C_1^\vee,C_1)$ given in \cite{DAHA2013}. By pulling back via $\zeta$, each $\H$-module is an $\Re$-module. Let $V$ denote a finite-dimensional irreducible $\H$-module. The set of $\Re$-submodules of $V$ forms a lattice under the inclusion partial order. 
The purpose of this paper is to classify the lattices that arise by this construction. In particular, we will see that the $\Re$-module $V$ is completely reducible if and only if $t_0$ is diagonalizable. Note that the classifications of finite-dimensional irreducible $\Re$-modules and $\H$-modules are given in \cite{SH:2019-1} and \cite{Huang:BImodule}, respectively.

The paper is organized as follows. In \S\ref{s:zeta} we give some preliminaries on $\Re$ and $\H$, as well as review the homomorphism $\zeta$ from $\Re$ into $\H$. In \S\ref{s:R&Hmodule} we lay the groundwork for the finite-dimensional irreducible $\Re$-modules and $\H$-modules. In \S\ref{s:lattice} we classify the lattices of $\Re$-submodules of finite-dimensional irreducible $\H$-modules. In \S\ref{s:proof} we end the paper with a summary of the classification and  its consequences.

\section{The universal Racah algebra and the universal additive DAHA of type $(C_1^\vee,C_1)$}\label{s:zeta}

\begin{defn}
[\!\!\cite{Levy1965,zhedanov1988,R&BI2015,HR2017}]
\label{defn:R}
The  {\it universal Racah algebra} $\Re$ is an $\F$-algebra defined by generators and relations in the following way. The generators are $A, B, C, D$ and the relations state that
\begin{gather}\label{r:D}
[A,B]=[B,C]=[C,A]=2D
\end{gather}
and each of
\begin{gather*}
[A,D]+AC-BA,
\qquad
[B,D]+BA-CB,
\qquad
[C,D]+CB-AC
\end{gather*}
commutes with $A,B,C,D$.  
\end{defn}

Let 
$$
\delta = A+B+C. 
$$

\begin{lem}\label{lem:delta}
\begin{enumerate}
\item The algebra $\Re$ is generated by $A,B,C$.

\item The algebra $\Re$ is generated by $A,B,\delta$.

\item The element $\delta$ is central in $\Re$.
\end{enumerate}
\end{lem}
\begin{proof}
(i): Immediate from (\ref{r:D}). 

(ii): Since $C=\delta-A-B$ and by (i) the statement (ii) follows.

(iii): By (\ref{r:D}) the element $\delta$ commutes with each of $A,B,C$. Hence (iii) follows by (i).
\end{proof}

\begin{defn}
[\!\!\!\cite{BI&NW2016,Groenevelt2007}] 
\label{defn:H}
The {\it universal additive DAHA {\rm (}double affine Hecke algebra{\rm)} $\H$ of type $(C_1^\vee,C_1)$} is an $\F$-algebra defined by generators and relations. The generators are $t_0,t_1,t_0^\vee,t_1^\vee$ and the relations state that  
\begin{gather}\label{t0+t1+t0'+t1'=-1}
t_0+t_1+t_0^\vee+t_1^\vee=-1
\end{gather}
and each of 
$
t_0^2,
t_1^2, 
t_0^{\vee 2}, 
t_1^{\vee 2}
$
commutes with $t_0,t_1,t_0^\vee,t_1^\vee$.
\end{defn}

Recall from \cite{tvz2012,BI2015,BI2016,BI2014,BI2014-2,BI2015-2,BI2017} that the Bannai--Ito algebra $\BI$ is an $\F$-algebra generated by $X,Y,Z$ and the relations assert that each of 
\begin{gather*}
\{X,Y\}-Z,
\qquad 
\{Y,Z\}-X,
\qquad 
\{Z,X\}-Y
\end{gather*}
is central in $\BI$. By \cite[Proposition 2]{BI&NW2016} there exists an $\F$-algebra isomorphism $\H\to \BI$ that sends 
\begin{eqnarray*}
t_0 &\mapsto& \frac{X+Y+Z}{2}-\frac{1}{4},
\\
t_1 &\mapsto& \frac{X-Y-Z}{2}-\frac{1}{4},
\\
t_0^\vee &\mapsto& \frac{Y-Z-X}{2}-\frac{1}{4},
\\
t_1^\vee &\mapsto& \frac{Z-X-Y}{2}-\frac{1}{4}.
\end{eqnarray*}

\begin{thm}
[\!\!\cite{R&BI2015,Huang:R<BI}]
\label{thm:hom}
There exists a unique $\F$-algebra homomorphism $\zeta:\Re\to \H$ that sends 
\begin{eqnarray*}
A &\mapsto & \frac{(t_1^\vee+t_0^\vee)(t_1^\vee+t_0^\vee+2)}{4},
\label{zeta:A}\\
B &\mapsto & \frac{(t_1+t_1^\vee)(t_1+t_1^\vee+2)}{4},
\label{zeta:B}\\
C &\mapsto & \frac{(t_0^\vee+t_1)(t_0^\vee+t_1+2)}{4},
\label{zeta:C}\\
\delta &\mapsto & \frac{t_0^2+t_1^2+t_0^{\vee 2}+t_1^{\vee 2}}{4}-\frac{t_0}{2}-\frac{3}{4}.
\label{zeta:delta}
\end{eqnarray*}
\end{thm}

By Theorem \ref{thm:hom} each $\H$-module is an $\Re$-module by pulling back via $\zeta$.

\section{Finite-dimensional irreducible $\Re$-modules and $\H$-modules}\label{s:R&Hmodule}

In \S\ref{s:Rmodule} we recall some results on the finite-dimensional irreducible $\Re$-modules from \cite{SH:2019-1}. In \S\ref{s:BImodule_even} and \S\ref{s:BImodule_odd} we rephrase some results on the finite-dimensional irreducible $\BI$-modules from \cite{Huang:BImodule} in terms of the $\H$-modules.

For convenience, we always let $a,b,c$ denote three scalars taken from $\F$ in the rest of this paper.

\subsection{Finite-dimensional irreducible $\Re$-modules}\label{s:Rmodule}

\begin{prop}
[\!\!\!\cite{SH:2019-1}]
\label{prop:Rd}
For any integer $d\geq 0$, there exists a $(d+1)$-dimensional $\Re$-module $R_d(a,b,c)$  satisfying the following conditions {\rm (i)}, {\rm (ii)}: 
\begin{enumerate}
\item There exists an $\F$-basis for $R_d(a,b,c)$ with respect to which the matrices representing $A$ and $B$ are 
$$
\begin{pmatrix}
\theta_0 & & &  &{\bf 0}
\\ 
1 &\theta_1 
\\
&1 &\theta_2
 \\
& &\ddots &\ddots
 \\
{\bf 0} & & &1 &\theta_d
\end{pmatrix},
\qquad 
\begin{pmatrix}
\theta_0^* &\varphi_1 &  & &{\bf 0}
\\ 
 &\theta_1^* &\varphi_2
\\
 &  &\theta_2^* &\ddots
 \\
 & & &\ddots &\varphi_d
 \\
{\bf 0}  & & & &\theta_d^*
\end{pmatrix}
$$
respectively, where 
\begin{align*}
\theta_i 
&=
(a+\textstyle\frac{d}{2}-i)
(a+\textstyle\frac{d}{2}-i+1)
\qquad 
(0\leq i\leq d),
\\
\theta^*_i 
&=
(b+\textstyle\frac{d}{2}-i)(b+\textstyle\frac{d}{2}-i+1)
\qquad 
(0\leq i\leq d),
\\
\varphi_i 
&=i(i-d-1)(a+b+c+\textstyle\frac{d}{2}-i+2)(a+b-c+\textstyle\frac{d}{2}-i+1)
\qquad 
(1\leq i\leq d).
\end{align*}

\item The element $\delta$ acts on $R_d (a,b,c)$ as scalar multiplication by
$$
\textstyle\frac{d}{2} (\textstyle\frac{d}{2}+1)+a(a+1)+b(b+1)+c(c+1).
$$
\end{enumerate} 
\end{prop}

\begin{prop}
[\!\!\!\cite{SH:2019-1}]
\label{prop:irr_R}
For any integer $d\geq 0$, the $\Re$-module $R_d(a,b,c)$ is irreducible if and only if $$
a+b+c+1,-a+b+c,a-b+c,a+b-c\not\in
\left
\{\frac{d}{2}-i\,\bigg|\,i=1,2,\ldots,d
\right\}.
$$
\end{prop}

\begin{thm}
[\!\!{\cite{SH:2019-1}}]
Let $d\geq 0$ denote an integer. If $V$ is a $(d+1)$-dimensional irreducible $\Re$-module then there exist $a,b,c\in \F$ such that $R_d(a,b,c)$ is isomorphic to $V$.
\end{thm}

\subsection{Even-dimensional irreducible $\H$-modules}\label{s:BImodule_even}

\begin{prop}
[\!\!\cite{Huang:BImodule}]
\label{prop:Ed} 
For any odd integer $d\geq 1$, there exists a $(d+1)$-dimensional $\H$-module $E_d(a,b,c)$ that has an $\mathbb F$-basis $\{v_i\}_{i=0}^d$ such that
\allowdisplaybreaks
\begin{align}
t_0 v_i &=
\left\{
\begin{array}{ll}
\displaystyle 
i (d-i+1) v_{i-1}
-\frac{d-2i+1}{2} v_{i}
\qquad &\hbox{for $i=2,4,\ldots,d-1$,}
\\
\displaystyle 
\frac{d-2i-1}{2} v_{i}
+ 
v_{i+1}
\qquad &\hbox{for $i=1,3,\ldots,d-2$,}
\end{array}
\right.
\label{t0:Ed-1}
\\
t_0 v_0 &=-\frac{d+1}{2} v_0,
\qquad 
t_0 v_d =-\frac{d+1}{2} v_d,
\label{t0:Ed-2}
\\
t_1 v_i &= 
\left\{
\begin{array}{ll}
\displaystyle 
i (i-d-1) v_{i-1}+a v_i+ v_{i+1}
\qquad 
&\hbox{for $i=2,4,\ldots,d-1$},
\\
-a v_i
\qquad 
&\hbox{for $i=1,3,\ldots,d$},
\end{array}
\right.
\label{t1:Ed-1}
\\
t_1 v_0&=a v_0+ v_1,
\label{t1:Ed-2}
\\
t_0^\vee v_i &=
\left\{
\begin{array}{ll}
b v_i 
\qquad 
&\hbox{for $i=0,2,\ldots,d-1$},
\\
-(\sigma+i)(\tau+i) v_{i-1}
-b v_i-v_{i+1}
\qquad 
&\hbox{for $i=1,3,\ldots,d-2$},
\end{array}
\right.
\label{t0vee:Ed-1}
\\
t_0^\vee v_d &=
-(\sigma+d)(\tau+d) v_{d-1}
-b v_{d},
\label{t0vee:Ed-2}
\\
t_1^\vee v_i &=
\left\{
\begin{array}{ll}
\displaystyle 
-\frac{\sigma+\tau+2i+2}{2} v_i
-v_{i+1}
\qquad 
&\hbox{for $i=0,2,\ldots,d-1$},
\\
\displaystyle 
(\sigma+i)(\tau+i) v_{i-1}+
\frac{\sigma+\tau+2i}{2} v_i
\qquad 
&\hbox{for $i=1,3,\ldots,d$},
\end{array}
\right.
\label{t1vee:Ed}
\end{align}
where 
\begin{gather*}
\sigma
=
a+b+c-\frac{d+1}{2},
\qquad 
\tau
=
a+b-c-\frac{d+1}{2}.
\end{gather*}
\end{prop}

\begin{lem}\label{lem:t^2_Ed}  
For any odd integer $d\geq 1$, the elements $t_0^2,t_1^2,t_0^{\vee 2}, t_1^{\vee 2}$ act on $E_d(a,b,c)$ as scalar multiplication by 
$
\frac{(d+1)^2}{4}, 
a^2,
b^2,
c^2
$
respectively.
\end{lem}
\begin{proof}
Apply Proposition \ref{prop:Ed} to evaluate the actions of $t_0^2,t_1^2,t_0^{\vee 2}, t_1^{\vee 2}$ on $E_d(a,b,c)$.
\end{proof}

\begin{prop}
[\!\!\cite{Huang:BImodule}]
\label{prop:irr_E} 
For any odd integer $d\geq 1$, the $\H$-module 
$
E_d(a,b,c)
$ 
is irreducible if and only if 
$$
a+b+c, -a+b+c, a-b+c, a+b-c\not\in \left\{
\displaystyle{\frac{d-1}{2}-i}\,\bigg|\,i=0,2,\ldots,d-1
\right\}.
$$
\end{prop}

The two-element set $\{\pm 1\}$ forms a group under multiplication and the group $\{\pm 1\}^2$ is isomorphic to the Klein $4$-group. 
Observe that there exists a unique $\{\pm 1\}^2$-action on $\H$ such that each $\e \in \{\pm 1\}^2$ acts on $\H$ as an $\F$-algebra automorphism in the following way:\\
\begin{table}[H]
\centering
\extrarowheight=3pt
\begin{tabular}{c|cccc}
$u$  &$t_0$ &$t_1$ &$t_0^\vee$ &$t_1^\vee$ 
\\

\midrule[1pt]

$u^{(1,1)}$ &$t_0$ &$t_1$ &$t_0^\vee$ &$t_1^\vee$ 
\\
$u^{(1,-1)}$ &$t_1$  &$t_0$ &$t_1^\vee$ &$t_0^\vee$
\\
$u^{(-1,1)}$ &$t_0^\vee$  &$t_1^\vee$ &$t_0$ &$t_1$
\\
$u^{(-1,-1)}$ &$t_1^\vee$  &$t_0^\vee$ &$t_1$ &$t_0$
\end{tabular}
\caption{The $\{\pm 1\}^2$-action on $\H$}\label{pm1-action}
\end{table}

\noindent 
Let $V$ denote an $\H$-module. 
For any $\F$-algebra automorphism $\e$ of $\H$ we define 
$$
V^\e
$$ 
to be the $\H$-module obtained by twisting $V$ via $\e$.

\begin{thm}
[\!\!\!\cite{Huang:BImodule}]
\label{thm:BImodule_even dim}
Let $d\geq 1$ denote an odd integer. If $V$ is a $(d+1)$-dimensional irreducible $\H$-module then there exist $a,b,c\in \F$ and $\e \in \{\pm 1\}^2$ such that $E_d(a,b,c)^\e$ is isomorphic to $V$.
\end{thm}

\subsection{Odd-dimensional irreducible $\H$-modules}\label{s:BImodule_odd}

\begin{prop}
[\!\!\cite{Huang:BImodule}]
\label{prop:Od}
For any even integer $d\geq 0$, there exists a $(d+1)$-dimensional $\H$-module $O_d(a,b,c)$ that has an $\mathbb F$-basis $\{v_i\}_{i=0}^d$ such that 
\allowdisplaybreaks
\begin{align*}
t_0 v_i
&=\left\{
\begin{array}{ll}
\displaystyle 
-i(\sigma+i)v_{i-1}+\frac{\sigma+2i}{2} v_i
\quad 
&\hbox{for $i=2,4,\ldots,d$},
\\
\displaystyle 
-\frac{\sigma+2i+2}{2} v_i+v_{i+1}
\quad 
&\hbox{for $i=1,3,\ldots,d-1$},
\end{array}
\right.
\\
t_0 v_0&=
\frac{\sigma}{2} v_0,
\\
t_1 v_i
&=\left\{
\begin{array}{ll}
\displaystyle 
i(\sigma+i) v_{i-1}+\frac{\lambda}{2} v_i+v_{i+1}
\quad 
&\hbox{for $i=2,4,\ldots,d-2$},
\\
\displaystyle 
-\frac{\lambda}{2}v_i
\quad 
&\hbox{for $i=1,3,\ldots,d-1$},
\end{array}
\right.
\\
t_1 v_0 &=
\frac{\lambda}{2} v_0+v_1,
\qquad 
t_1 v_d=
d(\sigma+d) v_{d-1}+\frac{\lambda}{2} v_d,
\\
t_0^\vee v_i
&=
\left\{
\begin{array}{ll}
\displaystyle 
\frac{\nu}{2} v_i
\quad 
&\hbox{for $i=0,2,\ldots,d$},
\\
\displaystyle 
(d-i+1)(\tau+i) v_{i-1}-\frac{\nu}{2} v_i-v_{i+1}
\quad 
&\hbox{for $i=1,3,\ldots,d-1$},
\end{array}
\right.
\\
t_1^\vee v_i 
&=
\left\{
\begin{array}{ll} 
\displaystyle 
\frac{2d+\mu-2i}{2} v_i
-v_{i+1}
\quad 
&\hbox{for $i=0,2,\ldots,d-2$},
\\
\displaystyle 
(i-d-1)(\tau+i) v_{i-1}
-\frac{2d+\mu-2i+2}{2} v_i
\quad 
&\hbox{for $i=1,3,\ldots,d-1$},
\end{array}
\right.
\\
t_1^\vee v_d 
&=
\frac{\mu}{2} v_d,
\end{align*}
where 
\begin{alignat*}{2}
\sigma
&=
a+b+c-\frac{d+1}{2},
\qquad 
\tau
&&=
a+b-c-\frac{d+1}{2},
\\
\lambda
&=
a-b-c-\frac{d+1}{2},
\qquad 
\mu
&&=
c-a-b-\frac{d+1}{2},
\\
\nu
&=
b-a-c-\frac{d+1}{2}.
\end{alignat*}
\end{prop}

\begin{lem}\label{lem:t^2_Od}
For any even integer $d\geq 0$, the elements $t_0^2,t_1^2,t_0^{\vee 2}, t_1^{\vee 2}$ act on $O_d(a,b,c)$ as scalar multiplication by 
\begin{gather*}
\displaystyle \left(
\frac{a+b+c}{2}-\frac{d+1}{4}
\right)^2,
\qquad 
\left(
\frac{a-b-c}{2}-\frac{d+1}{4}
\right)^2,
\\
\left(
\frac{c-a-b}{2}-\frac{d+1}{4}
\right)^2,
\qquad 
\left(
\frac{b-a-c}{2}-\frac{d+1}{4}
\right)^2,
\end{gather*}
respectively.
\end{lem}
\begin{proof}
Apply Proposition \ref{prop:Od} to evaluate the actions of $t_0^2,t_1^2,t_0^{\vee 2}, t_1^{\vee 2}$ on $O_d(a,b,c)$.
\end{proof}

\begin{prop}
[\!\!\!\cite{Huang:BImodule}]
\label{prop:irr_O}
For any even integer $d\geq 0$, the $\H$-module 
$
O_d(a,b,c)
$ is irreducible if and only if 
$$
a+b+c, a-b-c, -a+b-c, -a-b+c\not\in 
\left\{\frac{d+1}{2}-i\,\bigg|\,i=2,4,\ldots,d\right\}.
$$
\end{prop}

\begin{thm}
[\!\!\!\cite{Huang:BImodule}]
\label{thm:BImodule_odd dim}
Let $d\geq 0$ denote an even integer. If $V$ is a $(d+1)$-dimensional irreducible $\H$-module then there exist unique $a,b,c\in \F$ such that $O_d(a,b,c)$ is isomorphic to $V$.
\end{thm}

\section{The classification of lattices of $\Re$-submodules of finite-dimensional  irreducible $\H$-modules}\label{s:lattice}

In \S\ref{s:t0} we investigate the role of $t_0$ in the $\Re$-submodules of an $\H$-module. 
According to Theorems \ref{thm:BImodule_even dim} and \ref{thm:BImodule_odd dim} we may divide the lattices of $\Re$-submodules of the finite-dimensional irreducible $\H$-modules into five cases.
In \S\ref{s:lattice_Ed}--\ref{s:lattice_Od} we individually classify those lattices.

\subsection{The eigenspaces of $t_0$ as $\Re$-modules}\label{s:t0}

\begin{lem}\label{lem:t0}
The following equations hold in $\H$:
\begin{align*}
\{t_0+t_1,[t_1,t_0]\}&=0,
\\ 
\{t_0+t_0^\vee,[t_0^\vee,t_0]\}&=0,
\\
\{t_0+t_1^\vee,[t_1^\vee,t_0]\}&=0.
\end{align*}
\end{lem}
\begin{proof}
A direct calculation yields that 
\begin{gather}\label{t0+t1[t0t1]}
\{t_0+t_1,[t_1,t_0]\}=t_1^2t_0+t_1t_0^2-t_0^2t_1-t_0t_1^2.
\end{gather}
Since $t_0^2$ and $t_1^2$ are central in $\H$ by Definition \ref{defn:H}, the right-hand side of (\ref{t0+t1[t0t1]}) is zero. By similar arguments the other two equations follow.
\end{proof}

By \cite[Theorem 6.4]{Huang:R<BI}  the $\F$-algebra homomorphism $\zeta$ given in Theorem \ref{thm:hom} is injective. Thus the algebra $\Re$ can be  considered as a subalgebra of $\H$.

\begin{lem}\label{lem:t0_centralizer}
The element $t_0$ is in the centralizer of $\Re$ in $\H$. 
\end{lem}
\begin{proof}
By Lemma \ref{lem:delta} it suffices to show that $t_0$ commutes with $A$ and $B$. Any elements $x,y,z$ in a ring satisfy
\begin{gather}\label{[xy,z]}
[xy,z]=x[y,z]+[x,z]y.
\end{gather}
Applying (\ref{[xy,z]}) with $(x,y,z)=(t_0^\vee+t_1^\vee,t_0^\vee+t_1^\vee+2,t_0)$, 
the right-hand side of the resulting equation is 
\begin{gather}\label{[A,t0]_1}
(t_0^\vee+t_1^\vee)[t_0^\vee+t_1^\vee+2,t_0]+[t_0^\vee+t_1^\vee,t_0](t_0^\vee+t_1^\vee+2)
\end{gather}
and the left-hand side is $4[A,t_0]$ by Theorem \ref{thm:hom}. Using (\ref{t0+t1+t0'+t1'=-1}) yields that (\ref{[A,t0]_1}) is equal to $\{t_0+t_1,[t_1,t_0]\}$. 
Combined with Lemma \ref{lem:t0} we have $[A,t_0]=0$. Similarly the commutator $[B,t_0]$ is zero. The lemma follows.
\end{proof}

Given any $\H$-module $V$ and any $\theta\in \F$ we let 
$$
V(\theta)=\{v\in V\,|\, t_0v=\theta v\}.
$$

\begin{prop}\label{prop:t0eigenspace=Rmodule}
If $V$ is an $\H$-module then $V(\theta)$ is an $\Re$-submodule of $V$ for any $\theta\in \F$.
\end{prop}
\begin{proof}
For any $\theta\in \F$ it follows from Lemma \ref{lem:t0_centralizer} that 
 $V(\theta)$ is $x$-invariant for all $x\in \Re$.
\end{proof}

\begin{prop}\label{prop:irr_Rmodule_in_t0eigenspace}
Let $V$ denote a finite-dimensional irreducible $\H$-module. For any irreducible $\Re$-submodule $W$ of $V$, there exists a scalar $\theta\in \F$ such that $W\subseteq V(\theta)$.
\end{prop}
\begin{proof}
Recall from Lemma \ref{lem:delta}(iii) that $\delta$ is central in $\Re$.
Recall from Definition \ref{defn:H} that each of $t_0^2,t_1^2,t_0^{\vee 2}, t_1^{\vee 2}$ is central in $\H$. It follows from Schur's lemma that the action of $\delta$ on $W$ and the actions of $t_0^2,t_1^2,t_0^{\vee 2}, t_1^{\vee 2}$ on $V$ are scalar multiplication. By Theorem \ref{thm:hom} the element $t_0$ is an $\F$-linear combination of $1,\delta,t_0^2,t_1^2,t_0^{\vee 2}, t_1^{\vee 2}$. Hence $t_0$ acts on $W$ as scalar multiplication. The proposition follows.
\end{proof}

\subsection{The lattice of $\Re$-submodules of $E_d(a,b,c)$}\label{s:lattice_Ed}

Throughout \S\ref{s:lattice_Ed}--\S\ref{s:lattice_Ed(-1,-1)} we 
let $d\geq 1$ denote an odd integer and let $\{v_i\}_{i=0}^d$ denote the $\F$-basis for $E_d(a,b,c)$ from Proposition \ref{prop:Ed}. 
For notational convenience we set 
$$
\rho_i=c^2-\left(a+b-\frac{d+1}{2}+i\right)^2
\qquad \hbox{for $i=1,3,\ldots,d$}.
$$

\begin{lem}\label{lem2:iota_Ed}
The matrix representing $t_0$ with respect to the $\F$-basis 
\begin{gather*}
v_0,
\quad 
v_d,
\quad 
v_i-i v_{i-1}
\quad 
\hbox{for $i=2,4,\ldots,d-1$},
\quad 
v_i 
\quad 
\hbox{for $i=1,3,\ldots,d-2$}
\end{gather*}
for $E_d(a,b,c)$ is 
\begin{gather*}
\begin{pmatrix}
-\frac{d+1}{2} I_2 &\rvline &{\bf 0} &\rvline &{\bf 0}
\\
\hline
{\bf 0} &\rvline &-\frac{d+1}{2} I_{\frac{d-1}{2}} & \rvline &I_{\frac{d-1}{2}}
\\
\hline
{\bf 0} & \rvline &{\bf 0} & \rvline &\frac{d+1}{2}I_{\frac{d-1}{2}} 
\end{pmatrix}.
\end{gather*}
\end{lem}
\begin{proof}
Applying (\ref{t0:Ed-1}) and (\ref{t0:Ed-2}) it is routine to verify the lemma.
\end{proof}

\begin{lem}\label{lem:iota_Ed}
\begin{enumerate}
\item If $d=1$ then $t_0$ is diagonalizable on $E_d(a,b,c)$ with exactly one eigenvalue $-\frac{d+1}{2}$. 

\item If $d\geq 3$ then $t_0$ is diagonalizable on $E_d(a,b,c)$ with exactly two eigenvalues $\pm\frac{d+1}{2}$. 
\end{enumerate}
\end{lem}
\begin{proof}
Immediate from Lemma \ref{lem2:iota_Ed}.
\end{proof}

It follows from Proposition \ref{prop:t0eigenspace=Rmodule} that $E_d(a,b,c)(-\frac{d+1}{2})$ is an $\Re$-submodule of $E_d(a,b,c)$. We now go into the $\Re$-modules $E_d(a,b,c)(-\frac{d+1}{2})$ and $E_d(a,b,c)/E_d(a,b,c)(-\frac{d+1}{2})$.

\begin{lem}\label{lem3:iota_Ed}
$E_d(a,b,c)(-\frac{d+1}{2})$ is of dimension $\frac{d+3}{2}$ with the $\F$-basis 
\begin{gather}\label{basis:V(-d-1/2)}
v_0,
\quad 
v_d,
\quad 
v_i-i v_{i-1}
\quad 
\hbox{for $i=2,4,\ldots,d-1$}.
\end{gather}
\end{lem}
\begin{proof}
It is straightforward to verify the lemma by using Lemma \ref{lem2:iota_Ed}.
\end{proof}

\begin{lem}\label{lem:AB_Ed}
The actions of $A$ and $B$ on the $\H$-module $E_d(a,b,c)$ are as follows:
\begin{align*}
A v_i &=
\left\{
\begin{array}{ll}
\displaystyle
\theta_i v_i -\frac{1}{2} v_{i+1}+\frac{1}{4} v_{i+2}
\qquad 
&\hbox{for $i=0,2,\ldots,d-3$},
\\
\displaystyle
\theta_i v_i +\frac{1}{4} v_{i+2}
\qquad 
&\hbox{for $i=1,3,\ldots,d-2$},
\end{array}
\right.
\\
A v_{d-1}&=\theta_{d-1} v_{d-1} -\frac{1}{2} v_d,
\qquad 
A v_d = \theta_d v_d,
\\
B v_i &=
\left\{
\begin{array}{ll}
\displaystyle
\theta^*_i v_i +\frac{i(d-i+1)}{4}\rho_{i-1} v_{i-2}
\qquad 
&\hbox{for $i=2,4,\ldots,d-1$},
\\
\displaystyle
\theta^*_i v_i -\frac{\rho_i}{2} v_{i-1}+\frac{(i-1)(d-i+2)}{4}\rho_i v_{i-2}
\qquad 
&\hbox{for $i=3,5,\ldots,d$},
\end{array}
\right.
\\
B v_0 &=
\theta^*_0 v_0,
\qquad 
B v_1 =
\theta^*_1 v_1 -\frac{\rho_1}{2} v_0,
\end{align*}
where 
\begin{align*}
\theta_i
&=
\left(\frac{a}{2}-\frac{d-1}{4}+\left\lceil \frac{i}{2}\right\rceil\right)
\left(\frac{a}{2}-\frac{d+3}{4}+\left\lceil \frac{i}{2}\right\rceil\right)
\qquad (0\leq i\leq d),
\\
\theta_i^*
&=
\left(\frac{b}{2}-\frac{d-1}{4}+\left\lceil \frac{i}{2}\right\rceil\right)
\left(\frac{b}{2}-\frac{d+3}{4}+\left\lceil \frac{i}{2}\right\rceil\right)
\qquad (0\leq i\leq d).
\end{align*}
\end{lem}
\begin{proof}
It is routine to verify the lemma by applying Theorem \ref{thm:hom} and Proposition \ref{prop:Ed}.
\end{proof}

\begin{lem}\label{lem:AB_V(-d-1/2)}
The matrices representing $A$ and $B$ with respect to the $\F$-basis 
\begin{gather}\label{e:basis_V(-d-1/2)}
v_0, 
\quad 
\frac{1}{2^i}(v_i-i v_{i-1})
\quad \hbox{for $i=2,4,\ldots,d-1$},
\quad 
-\frac{(d+1)}{2^{d+1}} v_d
\end{gather}
for the $\Re$-module $E_d(a,b,c)(-\frac{d+1}{2})$ are 
$$
\begin{pmatrix}
\theta_0 & & &  &{\bf 0}
\\ 
1 &\theta_1 
\\
&1 &\theta_2
 \\
& &\ddots &\ddots
 \\
{\bf 0} & & &1 &\theta_\frac{d+1}{2}
\end{pmatrix},
\qquad 
\begin{pmatrix}
\theta_0^* &\varphi_1 &  & &{\bf 0}
\\ 
 &\theta_1^* &\varphi_2
\\
 &  &\theta_2^* &\ddots
 \\
 & & &\ddots &\varphi_{\frac{d+1}{2}}
 \\
{\bf 0}  & & & &\theta_\frac{d+1}{2}^*
\end{pmatrix}
$$
respectively, where 
\begin{align*}
\theta_i &=\frac{(2a-d+4i-3)(2a-d+4i+1)}{16}
\qquad (0\leq i\leq \textstyle\frac{d+1}{2}),
\\
\theta_i^* &= \frac{(2b-d+4i-3)(2b-d+4i+1)}{16}
\qquad (0\leq i\leq \textstyle\frac{d+1}{2}),
\\
\varphi_i &=
\frac{i(2i-d-3)(2a+2b+2c-d+4i-3)(2a+2b-2c-d+4i-3)}{32}
\qquad 
(1\leq i\leq \textstyle\frac{d+1}{2}).
\end{align*}
The element $\delta$ acts on the $\Re$-module $E_d(a,b,c)(-\frac{d+1}{2})$ as scalar multiplication by 
\begin{gather}\label{delta_V(-d-1/2)}
\frac{(d+1)(d+5)}{16}
+\frac{(a-1)(a+1)}{4}
+\frac{(b-1)(b+1)}{4}
+\frac{(c-1)(c+1)}{4}.
\end{gather}
\end{lem}
\begin{proof}
By Lemma \ref{lem3:iota_Ed} the vectors (\ref{e:basis_V(-d-1/2)}) are an $\F$-basis for $E_d(a,b,c)(-\frac{d+1}{2})$.
Applying Lemma \ref{lem:AB_Ed} a direct calculation yields the matrices representing $A$ and $B$ with respect to (\ref{e:basis_V(-d-1/2)}). By Theorem \ref{thm:hom} and Lemma \ref{lem:t^2_Ed} the element $\delta$ acts on $E_d(a,b,c)(-\frac{d+1}{2})$ as  scalar multiplication by (\ref{delta_V(-d-1/2)}). The lemma follows.
\end{proof}

\begin{prop}\label{prop:Rmodule_V(-d-1/2)}
The $\Re$-module $E_d(a,b,c)(-\frac{d+1}{2})$ is isomorphic to 
$$
R_{\frac{d+1}{2}}\left(
-\frac{a+1}{2},
-\frac{b+1}{2},
-\frac{c+1}{2}
\right).
$$
Moreover the $\Re$-module $E_d(a,b,c)(-\frac{d+1}{2})$ is irreducible provided that the $\H$-module $E_d(a,b,c)$ is irreducible.
\end{prop}
\begin{proof}
Set $(a',b',c',d')=(-\frac{a+1}{2},-\frac{b+1}{2},-\frac{c+1}{2},\frac{d+1}{2})$.
Comparing Proposition \ref{prop:Rd} with Lemma \ref{lem:AB_V(-d-1/2)} it follows that the $\Re$-module $E_d(a,b,c)(-\frac{d+1}{2})$ is isomorphic to $R_{d'}(a',b',c')$. Suppose that the $\H$-module $E_d(a,b,c)$ is irreducible. Using Proposition \ref{prop:irr_E} yields that 
\begin{gather*}
a'+b'+c'+1,
-a'+b'+c',
a'-b'+c',
a'+b'-c'
\not\in
\left\{
\frac{d'}{2}-i
\,\bigg|
\,
i=1,2,\ldots,d'
\right\}.
\end{gather*}
By Proposition \ref{prop:irr_R} the $\Re$-module $R_{d'}(a',b',c')$ is irreducible. The proposition follows.
\end{proof}

\begin{lem}\label{lem:AB_V/V(-d-1/2)}
Suppose that $d\geq 3$. Then the matrices representing $A$ and $B$ with respect to the $\F$-basis 
\begin{gather}\label{basis:V/V(-d-1/2)}
\frac{1}{2^{i-1}} v_i+E_d(a,b,c)(\textstyle-\frac{d+1}{2})
\qquad \hbox{for $i=1,3,\ldots,d-2$}
\end{gather}
for the $\Re$-module $E_d(a,b,c)/E_d(a,b,c)(-\frac{d+1}{2})$ are 
$$
\begin{pmatrix}
\theta_0 & & &  &{\bf 0}
\\ 
1 &\theta_1 
\\
&1 &\theta_2
 \\
& &\ddots &\ddots
 \\
{\bf 0} & & &1 &\theta_\frac{d-3}{2}
\end{pmatrix},
\qquad 
\begin{pmatrix}
\theta_0^* &\varphi_1 &  & &{\bf 0}
\\ 
 &\theta_1^* &\varphi_2
\\
 &  &\theta_2^* &\ddots
 \\
 & & &\ddots &\varphi_{\frac{d-3}{2}}
 \\
{\bf 0}  & & & &\theta_\frac{d-3}{2}^*
\end{pmatrix}
$$
respectively, where 
\begin{align*}
\theta_i &=
\frac{(2a-d+4i+5)(2a-d+4i+1)}{16}
\qquad (0\leq i\leq \textstyle\frac{d-3}{2}),
\\
\theta_i^* &=
\frac{(2b-d+4i+5)(2b-d+4i+1)}{16}
\qquad (0\leq i\leq \textstyle\frac{d-3}{2}),
\\
\varphi_i &=
\frac{i(2i-d+1)(2a+2b+2c-d+4i+1)(2a+2b-2c-d+4i+1)}{32}
\qquad (1\leq i\leq \textstyle\frac{d-3}{2}).
\end{align*}
The element $\delta$ acts on the $\Re$-module $E_d(a,b,c)/E_d(a,b,c)(-\frac{d+1}{2})$ as scalar multiplication by 
\begin{gather}\label{delta_V/V(-d-1/2)}
\frac{(d-3)(d+1)}{16}
+\frac{(a-1)(a+1)}{4}
+\frac{(b-1)(b+1)}{4}
+\frac{(c-1)(c+1)}{4}.
\end{gather}
\end{lem}
\begin{proof}
By Lemma \ref{lem3:iota_Ed} the cosets (\ref{basis:V/V(-d-1/2)}) are an $\F$-basis for $E_d(a,b,c)/E_d(a,b,c)(-\frac{d+1}{2})$.
Applying Lemma \ref{lem:AB_Ed} a direct calculation yields the matrices representing $A$ and $B$ with respect to (\ref{basis:V/V(-d-1/2)}). By Lemma \ref{lem2:iota_Ed} the element $t_0$ acts on $E_d(a,b,c)/E_d(a,b,c)(-\frac{d+1}{2})$ as scalar multiplication by $\frac{d+1}{2}$. Combined with Theorem \ref{thm:hom} and Lemma \ref{lem:t^2_Ed}, it follows that $\delta$ acts on $E_d(a,b,c)/E_d(a,b,c)(-\frac{d+1}{2})$ as scalar multiplication by (\ref{delta_V/V(-d-1/2)}). The lemma follows.
\end{proof}

\begin{prop}\label{prop:Rmodule_V/V(-d-1/2)}
Suppose that $d\geq 3$. Then the $\Re$-module $E_d(a,b,c)/E_d(a,b,c)(-\frac{d+1}{2})$ is  isomorphic to 
$$
R_{\frac{d-3}{2}}\left(
-\frac{a+1}{2},
-\frac{b+1}{2},
-\frac{c+1}{2}
\right).
$$
Moreover the $\Re$-module $E_d(a,b,c)/E_d(a,b,c)(-\frac{d+1}{2})$ is irreducible provided that the $\H$-module $E_d(a,b,c)$ is irreducible.
\end{prop}
\begin{proof}
Set $(a',b',c',d')=(
-\frac{a+1}{2},
-\frac{b+1}{2},
-\frac{c+1}{2},
\frac{d-3}{2})$.
Comparing Proposition \ref{prop:Rd} with
Lemma \ref{lem:AB_V/V(-d-1/2)} the quotient $\Re$-module $E_d(a,b,c)/E_d(a,b,c)(-\frac{d+1}{2})$ is isomorphic to $R_{d'}(a',b',c')$. Suppose that the $\H$-module $E_d(a,b,c)$ is irreducible. Using Proposition \ref{prop:irr_E} yields that 
\begin{gather*}
a'+b'+c'+1,
-a'+b'+c',
a'-b'+c',
a'+b'-c'
\not\in
\left\{
\frac{d'}{2}-i
\,\bigg|
\,
i=0,1,\ldots,d'+1
\right\}.
\end{gather*}
By Proposition \ref{prop:irr_R} the $\Re$-module $R_{d'}(a',b',c')$ is irreducible. The proposition follows.
\end{proof}

\begin{thm}\label{thm:Ed}
Assume that the $\H$-module $E_d(a,b,c)$ is irreducible. Then the following hold:
\begin{enumerate}
\item If $d=1$ then the $\Re$-module $E_d(a,b,c)$ is irreducible.

\item If $d\geq 3$ then
\begin{table}[H]
\begin{tikzpicture}[node distance=1.2cm]
 \node (0)                  {$\{0\}$};
 \node (1)  [above of=0]   {};
 \node (2)  [right of=1]   {};
 \node (3)  [left of=1]   {};
 \node (E1)  [right of=2]  {$E_d(a,b,c)(\frac{d+1}{2})$};
 \node (E2)  [left of=3]   {$E_d(a,b,c)(-\frac{d+1}{2})$};
 \node (V) [above of=1]  {$E_d(a,b,c)$};
 \draw (0)   -- (E1);
 \draw (0)   -- (E2);
 \draw (E1)   -- (V);
 \draw (E2)  -- (V);
\end{tikzpicture}
\end{table}
\noindent is the lattice of $\Re$-submodules of $E_d(a,b,c)$.
\end{enumerate}
\end{thm}
\begin{proof}
(i): Suppose that $d=1$. Then $E_d(a,b,c)=E_d(a,b,c)(-\frac{d+1}{2})$ 
by Lemma \ref{lem:iota_Ed}(i). It follows from Proposition \ref{prop:Rmodule_V(-d-1/2)} that the $\Re$-module $E_d(a,b,c)$ is irreducible.  The statement (i) follows.

(ii): Suppose that $d\geq 3$. Combining Propositions \ref{prop:Rmodule_V(-d-1/2)} and  \ref{prop:Rmodule_V/V(-d-1/2)} yields that
\begin{gather}\label{csEd-1}
\textstyle
\{0\}\subset E_d(a,b,c)(-\frac{d+1}{2})\subset E_d(a,b,c)
\end{gather}
is a composition series for the $\Re$-module $E_d(a,b,c)$. 
By Proposition \ref{prop:t0eigenspace=Rmodule} and Lemma \ref{lem:iota_Ed}(ii), $E_d(a,b,c)(\frac{d+1}{2})$ is a nonzero $\Re$-submodule of $E_d(a,b,c)$. 
By Jordan--H\"{o}lder theorem the sequence
\begin{gather}\label{csEd-2}
\textstyle
\{0\}\subset E_d(a,b,c)(\frac{d+1}{2})\subset E_d(a,b,c)
\end{gather}
is a composition series for the $\Re$-module $E_d(a,b,c)$. 
It follows from Proposition \ref{prop:irr_Rmodule_in_t0eigenspace} that there is no other irreducible $\Re$-submodule of $E_d(a,b,c)$. Hence (\ref{csEd-1}) and (\ref{csEd-2}) are the unique two composition series for the $\Re$-module $E_d(a,b,c)$. The statement (ii) follows.
\end{proof}

\subsection{The lattice of $\Re$-submodules of $E_d(a,b,c)^{(1,-1)}$}\label{s:lattice_Ed(1,-1)}

\begin{lem}\label{lem2:iota_Ed(1,-1)}
The matrix representing $t_0$ with respect to the $\F$-basis 
\begin{gather*}
v_1,
\quad 
v_{i+1}-i(d-i+1) v_{i-1}
\quad 
\hbox{for $i=2,4,\ldots,d-1$},
\quad 
v_i 
\quad 
\hbox{for $i=0,2,\ldots,d-1$}
\end{gather*}
for $E_d(a,b,c)^{(1,-1)}$ is 
\begin{gather*}
\begin{pmatrix}
-a I_{\frac{d+1}{2}} & \rvline & I_{\frac{d+1}{2}}
\\
\hline
{\bf 0} & \rvline &a I_{\frac{d+1}{2}} 
\end{pmatrix}.
\end{gather*}
\end{lem}
\begin{proof}
By Table \ref{pm1-action} the action of $t_0$ on $E_d(a,b,c)^{(1,-1)}$ corresponds to the action of $t_1$ on $E_d(a,b,c)$. 
By (\ref{t1:Ed-1}) and (\ref{t1:Ed-2}) it is routine to verify the lemma.
\end{proof}

\begin{lem}\label{lem:iota_Ed(1,-1)}
\begin{enumerate}
\item If $a=0$ then $t_0$ is not diagonalizable on $E_d(a,b,c)^{(1,-1)}$ with exactly one eigenvalue $0$.

\item If $a\not=0$ then $t_0$ is diagonalizable on $E_d(a,b,c)^{(1,-1)}$ with exactly two eigenvalues $\pm a$. 
\end{enumerate}
\end{lem}
\begin{proof}
Immediate from Lemma \ref{lem2:iota_Ed(1,-1)}.
\end{proof}

\begin{lem}\label{lem3:iota_Ed(1,-1)}
$E_d(a,b,c)^{(1,-1)}(-a)$ is of dimension $\frac{d+1}{2}$ with the 
$\F$-basis 
\begin{gather*}
v_i
\qquad 
\hbox{for $i=1,3,\ldots,d$}.
\end{gather*}
\end{lem}
\begin{proof}
Immediate from Lemma \ref{lem2:iota_Ed(1,-1)}.
\end{proof}

\begin{lem}\label{lem:AB_Ed(1,-1)}
The actions of $A$ and $B$ on the $\H$-module $E_d(a,b,c)^{(1,-1)}$ are as follows:
\begin{align*}
A v_i &=
\left\{
\begin{array}{ll}
\displaystyle
\theta_i v_i -\frac{1}{2} v_{i+1}+\frac{1}{4} v_{i+2}
\qquad 
&\hbox{for $i=0,2,\ldots,d-3$},
\\
\displaystyle
\theta_i v_i +\frac{1}{4} v_{i+2}
\qquad 
&\hbox{for $i=1,3,\ldots,d-2$},
\end{array}
\right.
\\
A v_{d-1}&=\theta_{d-1} v_{d-1} -\frac{1}{2} v_d,
\qquad 
A v_d = \theta_d v_d,
\\
B v_i &=
\left\{
\begin{array}{ll}
\displaystyle
\theta^*_i v_i +\frac{i(d-i+1)}{2} v_{i-1}+\frac{i(d-i+1)}{4}\rho_{i-1} v_{i-2}
\qquad 
&\hbox{for $i=2,4,\ldots,d-1$},
\\
\displaystyle
\theta^*_i v_i+\frac{(i-1)(d-i+2)}{4} \rho_i v_{i-2}
\qquad 
&\hbox{for $i=3,5,\ldots, d$},
\end{array}
\right.
\\
B v_0 &= \theta^*_0 v_0,
\qquad 
B v_1 = \theta^*_1 v_1,
\end{align*}
where 
\begin{align*}
\theta_i
&=
\left(\frac{a}{2}-\frac{d-1}{4}+\left\lceil \frac{i}{2}\right\rceil\right)
\left(\frac{a}{2}-\frac{d+3}{4}+\left\lceil \frac{i}{2}\right\rceil\right)
\qquad (0\leq i\leq d),
\\
\theta_i^*
&=
\left(\frac{b}{2}-\frac{d-3}{4}+\left\lfloor \frac{i}{2}\right\rfloor\right)
\left(\frac{b}{2}-\frac{d+1}{4}+\left\lfloor \frac{i}{2}\right\rfloor\right)
\qquad (0\leq i\leq d).
\end{align*}
\end{lem}
\begin{proof}
By Theorem \ref{thm:hom} and Table \ref{pm1-action} the actions of $A$ and $B$ on $E_d(a,b,c)^{(1,-1)}$ correspond to the actions of 
$$
\frac{(t_0^\vee+t_1^\vee)(t_0^\vee+t_1^\vee+2)}{4},
\qquad 
\frac{(t_0+t_0^\vee)(t_0+t_0^\vee+2)}{4}
$$
on $E_d(a,b,c)$, respectively.
Applying Proposition \ref{prop:Ed} it is routine to verify the lemma.
\end{proof}

\begin{lem}\label{lem:AB_V(-a)}
The matrices representing $A$ and $B$ with respect to the $\F$-basis 
\begin{gather}\label{e:basis_V(-a)}
\frac{1}{2^{i-1}} v_i
\qquad 
\hbox{for $i=1,3,\ldots,d$}
\end{gather}
for the $\Re$-module $E_d(a,b,c)^{(1,-1)}(-a)$ are 
$$
\begin{pmatrix}
\theta_0 & & &  &{\bf 0}
\\ 
1 &\theta_1 
\\
&1 &\theta_2
 \\
& &\ddots &\ddots
 \\
{\bf 0} & & &1 &\theta_\frac{d-1}{2}
\end{pmatrix},
\qquad 
\begin{pmatrix}
\theta_0^* &\varphi_1 &  & &{\bf 0}
\\ 
 &\theta_1^* &\varphi_2
\\
 &  &\theta_2^* &\ddots
 \\
 & & &\ddots &\varphi_{\frac{d-1}{2}}
 \\
{\bf 0}  & & & &\theta_\frac{d-1}{2}^*
\end{pmatrix}
$$
respectively, where 
\begin{align*}
\theta_i
&=\frac{(2a-d+4i+1)(2a-d+4i+5)}{16} 
\qquad (0\leq i\leq \textstyle\frac{d-1}{2}),
\\
\theta_i^*
&=\frac{(2b-d+4i-1)(2b-d+4i+3)}{16}
\qquad (0\leq i\leq \textstyle\frac{d-1}{2}),
\\
\varphi_i &=
\frac{i(2i-d-1)(2a+2b+2c-d+4i+1)(2a+2b-2c-d+4i+1)}{32}
\qquad 
(1\leq i\leq \textstyle\frac{d-1}{2}).
\end{align*}
The element $\delta$ acts on the $\Re$-module $E_d(a,b,c)^{(1,-1)}(-a)$ as scalar multiplication by 
\begin{gather}\label{delta_V(-a)}
\frac{(d-1)(d+3)}{16}+\frac{a(a+2)}{4}
+\frac{(b-1)(b+1)}{4}+\frac{(c-1)(c+1)}{4}.
\end{gather}
\end{lem}
\begin{proof}
By Lemma \ref{lem3:iota_Ed(1,-1)} the vectors (\ref{e:basis_V(-a)}) are an $\F$-basis for $E_d(a,b,c)^{(1,-1)}(-a)$. 
Applying Lemma \ref{lem:AB_Ed(1,-1)}  a direct calculation yields the matrices representing $A$ and $B$ with respect to (\ref{e:basis_V(-a)}). Applying Theorem \ref{thm:hom} and Lemma \ref{lem:t^2_Ed} yields that $\delta$ acts on $E_d(a,b,c)^{(1,-1)}(-a)$ as scalar multiplication by (\ref{delta_V(-a)}). The lemma follows.
\end{proof}

\begin{prop}\label{prop:Rmodule_V(-a)}
The $\Re$-module $E_d(a,b,c)^{(1,-1)}(-a)$ is isomorphic to $$
R_{\frac{d-1}{2}}\left(
-\frac{a}{2}-1,
-\frac{b+1}{2},
-\frac{c+1}{2}
\right).
$$
Moreover the $\Re$-module $E_d(a,b,c)^{(1,-1)}(-a)$ is irreducible if the $\H$-module $E_d(a,b,c)^{(1,-1)}$ is irreducible.
\end{prop}
\begin{proof}
Set $(a',b',c',d')=(-\frac{a}{2}-1,
-\frac{b+1}{2},
-\frac{c+1}{2},\frac{d-1}{2})$.
Comparing Proposition \ref{prop:Rd} with Lemma \ref{lem:AB_V(-a)} it follows that the $\Re$-module $E_d(a,b,c)^{(1,-1)}(-a)$ is isomorphic to $R_{d'}(a',b',c')$. Suppose that the $\H$-module $E_d(a,b,c)^{(1,-1)}$ is irreducible. Using Proposition \ref{prop:irr_E} yields that 
$$
a'+b'+c'+1,
a'-b'+c',
a'+b'-c'
\not\in
\left\{
\frac{d'}{2}-i
\,\bigg|
\,
i=1,2,\ldots,d'+1
\right\}
$$
and 
\begin{gather*}
-a'+b'+c'
\not\in
\left\{
\frac{d'}{2}-i
\,\bigg|
\,
i=0,1,\ldots,d'
\right\}.
\end{gather*}
By Proposition \ref{prop:irr_R} the $\Re$-module $R_{d'}(a',b',c')$ is irreducible. The proposition follows.
\end{proof}

\begin{lem}\label{lem:AB_V/V(-a)}
The matrices representing $A$ and $B$ with respect to the $\F$-basis 
\begin{gather}\label{basis_V/V(-a)}
\frac{1}{2^i} v_i+E_d(a,b,c)^{(1,-1)}(-a)
\qquad 
\hbox{for $i=0,2,\ldots,d-1$}
\end{gather}
for the $\Re$-module $E_d(a,b,c)^{(1,-1)}/E_d(a,b,c)^{(1,-1)}(-a)$ are 
$$
\begin{pmatrix}
\theta_0 & & &  &{\bf 0}
\\ 
1 &\theta_1 
\\
&1 &\theta_2 
 \\
& &\ddots &\ddots
 \\
{\bf 0} & & &1 &\theta_\frac{d-1}{2}
\end{pmatrix},
\qquad 
\begin{pmatrix}
\theta_0^* &\varphi_1 &  & &{\bf 0}
\\ 
 &\theta_1^* &\varphi_2
\\
 &  &\theta_2^* &\ddots
 \\
 & & &\ddots &\varphi_{\frac{d-1}{2}}
 \\
{\bf 0}  & & & &\theta_\frac{d-1}{2}^*
\end{pmatrix}
$$
respectively, where 
\begin{align*}
\theta_i
&=\frac{(2a-d+4i-3)(2a-d+4i+1)}{16} 
\qquad (0\leq i\leq \textstyle\frac{d-1}{2}),
\\
\theta_i^*
&=\frac{(2b-d+4i-1)(2b-d+4i+3)}{16}
\qquad (0\leq i\leq \textstyle\frac{d-1}{2}),
\\
\varphi_i &=
\frac{i(2i-d-1)(2a+2b+2c-d+4i-3)(2a+2b-2c-d+4i-3)}{32}
\qquad 
(1\leq i\leq \textstyle\frac{d-1}{2}).
\end{align*}
The element $\delta$ acts on the $\Re$-module $E_d(a,b,c)^{(1,-1)}/E_d(a,b,c)^{(1,-1)}(-a)$ as scalar multiplication by
\begin{gather}\label{delta_V/V(-a)}
\frac{(d-1)(d+3)}{16}+\frac{a(a-2)}{4}
+\frac{(b-1)(b+1)}{4}+\frac{(c-1)(c+1)}{4}.
\end{gather}
\end{lem}
\begin{proof}
By Lemma \ref{lem3:iota_Ed(1,-1)} the cosets (\ref{basis_V/V(-a)}) are an $\F$-basis for $E_d(a,b,c)^{(1,-1)}/E_d(a,b,c)^{(1,-1)}(-a)$.
Applying Lemma \ref{lem:AB_Ed(1,-1)} a direct calculation yields the matrices representing $A$ and $B$ with respect to (\ref{basis_V/V(-a)}). 
By Lemma \ref{lem2:iota_Ed(1,-1)} the element $t_0$ acts on $E_d(a,b,c)^{(1,-1)}/E_d(a,b,c)^{(1,-1)}(-a)$ as scalar multiplication by $a$. Combined with Theorem \ref{thm:hom} and Lemma \ref{lem:t^2_Ed}, the element $\delta$ acts on $E_d(a,b,c)^{(1,-1)}/E_d(a,b,c)^{(1,-1)}(-a)$ as scalar multiplication by (\ref{delta_V/V(-a)}). The lemma follows.
\end{proof}

\begin{prop}\label{prop:Rmodule_V/V(-a)}
The $\Re$-module $E_d(a,b,c)^{(1,-1)}/E_d(a,b,c)^{(1,-1)}(-a)$ is isomorphic to $$
R_{\frac{d-1}{2}}\left(
-\frac{a}{2},
-\frac{b+1}{2},
-\frac{c+1}{2}
\right).
$$
Moreover the $\Re$-module 
$
E_d(a,b,c)^{(1,-1)}/ E_d(a,b,c)^{(1,-1)}(-a)
$ 
is irreducible provided that the $\H$-module 
$E_d(a,b,c)^{(1,-1)}$ is irreducible.
\end{prop}
\begin{proof}
Let $(a',b',c',d')=(-\frac{a}{2},
-\frac{b+1}{2},
-\frac{c+1}{2},
\frac{d-1}{2})$. 
Comparing Proposition \ref{prop:Rd} with Lemma \ref{lem:AB_V/V(-a)} yields that the quotient $\Re$-module $E_d(a,b,c)^{(1,-1)}/E_d(a,b,c)^{(1,-1)}(-a)$ is isomorphic to $R_{d'}(a',b',c')$. Suppose that the $\H$-module $E_d(a,b,c)^{(1,-1)}$ is irreducible.
Using Proposition \ref{prop:irr_E} yields that 
$$
a'+b'+c'+1',a'-b'+c',a'+b'-c'\not\in
\left\{
\frac{d'}{2}-i\,\bigg|\, i=0,1,\ldots,d'
\right\}
$$
and 
$$
-a'+b'+c'\not\in
\left\{
\frac{d'}{2}-i\,\bigg|\, i=1,2,\ldots,d'+1
\right\}.
$$
By Proposition \ref{prop:irr_R} the $\Re$-module $R_{d'}(a',b',c')$ is irreducible. The proposition follows.
\end{proof}

\begin{thm}\label{thm:Ed(1,-1)}
Assume that the $\H$-module $E_d(a,b,c)^{(1,-1)}$ is irreducible. Then the following hold:
\begin{enumerate}
\item If $a=0$ then 
\begin{table}[H]
\begin{tikzpicture}[node distance=1.2cm]
 \node (0)                  {$\{0\}$};
 \node (E)  [above of=0]   {$E_d(a,b,c)^{(1,-1)}(0)$};
 \node (V)  [above of=E]   {$E_d(a,b,c)^{(1,-1)}$};
 \draw (0)   -- (E);
 \draw (E)  -- (V);
\end{tikzpicture}
\end{table}
\noindent is the lattice of $\Re$-submodules of $E_d(a,b,c)^{(1,-1)}$.

\item If $a\not=0$ then 
 \begin{table}[H]
\begin{tikzpicture}[node distance=1.2cm]
 \node (0)                  {$\{0\}$};
 \node (1)  [above of=0]   {};
 \node (2)  [right of=1]   {};
 \node (3)  [left of=1]   {};
 \node (E1)  [right of=2]  {$E_d(a,b,c)^{(1,-1)}(a)$};
 \node (E2)  [left of=3]   {$E_d(a,b,c)^{(1,-1)}(-a)$};
 \node (V) [above of=1]  {$E_d(a,b,c)^{(1,-1)}$};
 \draw (0)   -- (E1);
 \draw (0)   -- (E2);
 \draw (E1)   -- (V);
 \draw (E2)  -- (V);
\end{tikzpicture}
\end{table}
\noindent is the lattice of $\Re$-submodules of $E_d(a,b,c)^{(1,-1)}$.
\end{enumerate}
\end{thm}
\begin{proof}

(i): Suppose that $a=0$. Combining Propositions \ref{prop:Rmodule_V(-a)} and \ref{prop:Rmodule_V/V(-a)} yields that 
\begin{gather}\label{cs1:Ed(1,-1)}
\{0\}\subset 
E_d(a,b,c)^{(1,-1)}(0)
\subset 
E_d(a,b,c)^{(1,-1)}
\end{gather}
is a composition series for the $\Re$-module $E_d(a,b,c)^{(1,-1)}$. By Proposition \ref{prop:irr_Rmodule_in_t0eigenspace} and Lemma \ref{lem:iota_Ed(1,-1)}(i) every irreducible $\Re$-submodule of $E_d(a,b,c)^{(1,-1)}$ is contained in $E_d(a,b,c)^{(1,-1)}(0)$. Hence (\ref{cs1:Ed(1,-1)}) is the unique composition series for the $\Re$-module $E_d(a,b,c)^{(1,-1)}$. Therefore (i) follows.

(ii): Similar to the proof of Theorem \ref{thm:Ed}(ii). 
\end{proof}

\subsection{The lattice of $\Re$-submodules of $E_d(a,b,c)^{(-1,1)}$}\label{s:lattice_Ed(-1,1)}

\begin{lem}\label{lem2:iota_Ed(-1,1)}
Assume that the $\H$-module $E_d(a,b,c)^{(-1,1)}$ is irreducible. 
Then 
\begin{equation}\label{basis:Ed(-1,1)}
\textstyle 
\rho_{i-1} v_{i-2}-v_i
\quad 
\hbox{for $i=2,4,\ldots,d-1$},
\quad 
\rho_d v_{d-1},
\quad 
v_i 
\quad 
\hbox{for $i=1,3,\ldots,d$}
\end{equation}
form an $\F$-basis for $E_d(a,b,c)^{(-1,1)}$.
The matrix representing $t_0$ with respect to the $\F$-basis {\rm (\ref{basis:Ed(-1,1)})} for $E_d(a,b,c)^{(-1,1)}$ is 
\begin{gather*}
\begin{pmatrix}
bI_{\frac{d+1}{2}} & \rvline &I_{\frac{d+1}{2}}
\\
\hline
{\bf 0} & \rvline & -b I_{\frac{d+1}{2}}
\end{pmatrix}.
\end{gather*}
\end{lem}
\begin{proof}
It follows from Proposition \ref{prop:irr_E} that $\rho_i\not=0$ for all $i=1,3,\ldots,d$. Hence (\ref{basis:Ed(-1,1)}) is an $\F$-basis for $E_d(a,b,c)^{(-1,1)}$. 
By Table \ref{pm1-action} the action of $t_0$ on $E_d(a,b,c)^{(-1,1)}$ corresponds to the action of $t_0^\vee$ on $E_d(a,b,c)$.  
Using (\ref{t0vee:Ed-1}) and (\ref{t0vee:Ed-2}) it is routine to verify the lemma.
\end{proof}

\begin{lem}\label{lem:iota_Ed(-1,1)}
Assume that the $\H$-module $E_d(a,b,c)^{(-1,1)}$ is irreducible. Then the following hold:
\begin{enumerate}
\item If $b=0$ then $t_0$ is not diagonalizable on $E_d(a,b,c)^{(-1,1)}$ with exactly one eigenvalue $0$.

\item If $b\not=0$ then $t_0$ is diagonalizable on $E_d(a,b,c)^{(-1,1)}$ with exactly two eigenvalues $\pm b$. 
\end{enumerate}
\end{lem}
\begin{proof}
Immediate from Lemma \ref{lem2:iota_Ed(-1,1)}.
\end{proof}

\begin{lem}\label{lem3:iota_Ed(-1,1)}
If the $\H$-module $E_d(a,b,c)^{(-1,1)}$ is irreducible then 
$E_d(a,b,c)^{(-1,1)}(b)$ is of dimension $\frac{d+1}{2}$ with the $\F$-basis 
\begin{gather*}
v_i
\qquad 
\hbox{for $i=0,2,\ldots,d-1$}.
\end{gather*}
\end{lem}
\begin{proof}
Immediate from Lemma \ref{lem2:iota_Ed(-1,1)}.
\end{proof}

\begin{lem}\label{lem:AB_Ed(-1,1)}
The actions of $A$ and $B$ on the $\H$-module $E_d(a,b,c)^{(-1,1)}$ are as follows:
\begin{align*}
A v_i &=
\left\{
\begin{array}{ll}
\displaystyle
\theta_i v_{i}+\frac{1}{4} v_{i+2} 
\qquad 
&\hbox{for $i=0,2,\ldots,d-3$},
\\
\displaystyle
\theta_i v_i+\frac{1}{2}v_{i+1}+\frac{1}{4} v_{i+2}
\qquad 
&\hbox{for $i=1,3,\ldots,d-2$},
\end{array}
\right.
\\
A v_{d-1}&=\theta_{d-1} v_{d-1},
\qquad 
A v_d=\theta_d v_d,
\\
B v_i &=
\left\{
\begin{array}{ll}
\displaystyle
\theta^*_i v_i +\frac{i(d-i+1)}{4}\rho_{i-1} v_{i-2}
\qquad 
&\hbox{for $i=2,4,\ldots,d-1$},
\\
\displaystyle
\theta^*_i v_i -\frac{\rho_i}{2} v_{i-1}+\frac{(i-1)(d-i+2)}{4}\rho_i v_{i-2}
\qquad 
&\hbox{for $i=3,5,\ldots,d$},
\end{array}
\right.
\\
B v_0 &=
\theta^*_0 v_0,
\qquad 
B v_1 =
\theta^*_1 v_1 -\frac{\rho_1}{2} v_0,
\end{align*}
where 
\begin{align*}
\theta_i
&=
\left(\frac{a}{2}-\frac{d-3}{4}+\left\lfloor \frac{i}{2}\right\rfloor\right)
\left(\frac{a}{2}-\frac{d+1}{4}+\left\lfloor \frac{i}{2}\right\rfloor\right)
\qquad (0\leq i\leq d),
\\
\theta_i^*
&=
\left(\frac{b}{2}-\frac{d-1}{4}+\left\lceil \frac{i}{2}\right\rceil\right)
\left(\frac{b}{2}-\frac{d+3}{4}+\left\lceil \frac{i}{2}\right\rceil\right)
\qquad (0\leq i\leq d).
\end{align*}
\end{lem}
\begin{proof}
By Theorem \ref{thm:hom} and Table \ref{pm1-action} the actions of $A$ and $B$ on $E_d(a,b,c)^{(1,-1)}$ correspond to the actions of 
$$
\frac{(t_0+t_1)(t_0+t_1+2)}{4},
\qquad 
\frac{(t_1+t_1^\vee)(t_1+t_1^\vee+2)}{4}
$$
on $E_d(a,b,c)$, respectively. Using Proposition \ref{prop:Ed} 
it is routine to verify the lemma.
\end{proof}

\begin{lem}\label{lem:AB_V(b)}
Assume that the $\H$-module $E_d(a,b,c)^{(-1,1)}$ is irreducible. 
Then the matrices representing $A$ and $B$ with respect to the $\F$-basis 
\begin{gather}\label{e:basis_V(b)}
\frac{1}{2^i}v_i 
\qquad \hbox{for $i=0,2,\ldots,d-1$}
\end{gather}
for the $\Re$-module $E_d(a,b,c)^{(-1,1)}(b)$ are 
$$
\begin{pmatrix}
\theta_0 & & &  &{\bf 0}
\\ 
1 &\theta_1 
\\
&1 &\theta_2 
 \\
& &\ddots &\ddots
 \\
{\bf 0} & & &1 &\theta_\frac{d-1}{2}
\end{pmatrix},
\qquad 
\begin{pmatrix}
\theta_0^* &\varphi_1 &  & &{\bf 0}
\\ 
 &\theta_1^* &\varphi_2
\\
 &  &\theta_2^* &\ddots
 \\
 & & &\ddots &\varphi_{\frac{d-1}{2}}
 \\
{\bf 0}  & & & &\theta_\frac{d-1}{2}^*
\end{pmatrix}
$$
respectively, where
\allowdisplaybreaks 
\begin{align*}
\theta_i &=
\frac{(2a-d+4i-1)(2a-d+4i+3)}{16}
\qquad (0\leq i\leq \textstyle\frac{d-1}{2}),
\\
\theta_i^* &= 
\frac{(2b-d+4i-3)(2b-d+4i+1)}{16}
\qquad (0\leq i\leq \textstyle\frac{d-1}{2}),
\\
\varphi_i &=
\frac{i(2i-d-1)(2a+2b+2c+4i-d-3)(2a+2b-2c+4i-d-3)}{32}
\qquad 
(1\leq i\leq \textstyle\frac{d-1}{2}).
\end{align*}
The element $\delta$ acts on the $\Re$-module $E_d(a,b,c)^{(-1,1)}(b)$ as scalar multiplication by 
\begin{gather}\label{delta_V(b)}
\frac{(d-1)(d+3)}{16}+\frac{(a-1)(a+1)}{4}
+\frac{b(b-2)}{4}+\frac{(c-1)(c+1)}{4}.
\end{gather}
\end{lem}
\begin{proof}
By Lemma \ref{lem3:iota_Ed(-1,1)} the vectors (\ref{e:basis_V(b)}) are an $\F$-basis for $E_d(a,b,c)^{(-1,1)}(b)$. 
Applying Lemma \ref{lem:AB_Ed(-1,1)} a straightforward calculation yields the matrices representing $A$ and $B$ with respect to (\ref{e:basis_V(b)}). By Theorem \ref{thm:hom} and Lemma \ref{lem:t^2_Ed} the element $\delta$ acts on $E_d(a,b,c)^{(-1,1)}(b)$ as scalar multiplication by (\ref{delta_V(b)}). The lemma follows.
\end{proof}

\begin{prop}\label{prop:Rmodule_V(b)}
Assume that the $\H$-module $E_d(a,b,c)^{(-1,1)}$ is irreducible. 
The $\Re$-module $E_d(a,b,c)^{(-1,1)}(b)$ is isomorphic to $$
R_{\frac{d-1}{2}}\left(
-\frac{a+1}{2},
-\frac{b}{2},
-\frac{c+1}{2}
\right).
$$
Moreover the $\Re$-module $E_d(a,b,c)^{(-1,1)}(b)$ is irreducible.
\end{prop}
\begin{proof}
Set $(a',b',c',d')=(-\frac{a+1}{2},
-\frac{b}{2},
-\frac{c+1}{2},\frac{d-1}{2})$.
Comparing Proposition \ref{prop:Rd} with Lemma \ref{lem:AB_V(b)} yields that the $\Re$-module $E_d(a,b,c)^{(-1,1)}(b)$ is isomorphic to $R_{d'}(a',b',c')$. It follows from Proposition \ref{prop:irr_E} that 
$$
a'+b'+c'+1,
-a'+b'+c',
a'+b'-c'
\not\in
\left\{
\frac{d'}{2}-i
\,\bigg|
\,
i=0,1,\ldots,d'
\right\}
$$
and 
\begin{gather*}
a'-b'+c'
\not\in
\left\{
\frac{d'}{2}-i
\,\bigg|
\,
i=1,2,\ldots,d'+1
\right\}.
\end{gather*}
By Proposition \ref{prop:irr_R} the $\Re$-module $R_{d'}(a',b',c')$ is irreducible. The proposition follows.
\end{proof}

\begin{lem}\label{lem:AB_V/V(b)}
Assume that the $\H$-module $E_d(a,b,c)^{(-1,1)}$ is irreducible. 
Then the matrices representing $A$ and $B$ with respect to the $\F$-basis 
\begin{gather}\label{e:basis_V/V(b)}
\frac{1}{2^{i-1}}v_i +E_d(a,b,c)^{(-1,1)}(b)
\qquad \hbox{for $i=1,3,\ldots,d$}
\end{gather}
for the $\Re$-module $E_d(a,b,c)^{(-1,1)}/E_d(a,b,c)^{(-1,1)}(b)$ are 
$$
\begin{pmatrix}
\theta_0 & & &  &{\bf 0}
\\ 
1 &\theta_1 
\\
&1 &\theta_2 
 \\
& &\ddots &\ddots
 \\
{\bf 0} & & &1 &\theta_\frac{d-1}{2}
\end{pmatrix},
\qquad 
\begin{pmatrix}
\theta_0^* &\varphi_1 &  & &{\bf 0}
\\ 
 &\theta_1^* &\varphi_2
\\
 &  &\theta_2^* &\ddots
 \\
 & & &\ddots &\varphi_{\frac{d-1}{2}}
 \\
{\bf 0}  & & & &\theta_\frac{d-1}{2}^*
\end{pmatrix}
$$
respectively, where 
\begin{align*}
\theta_i &=
\frac{(2a-d+4i-1)(2a-d+4i+3)}{16}
\qquad (0\leq i\leq \textstyle\frac{d-1}{2}),
\\
\theta_i^* &= 
\frac{(2b-d+4i+1)(2b-d+4i+5)}{16}
\qquad (0\leq i\leq \textstyle\frac{d-1}{2}),
\\
\varphi_i &=
\frac{i(2i-d-1)(2a+2b+2c+4i-d+1)(2a+2b-2c+4i-d+1)}{32}
\qquad 
(1\leq i\leq \textstyle\frac{d-1}{2}).
\end{align*}
The element $\delta$ acts on the $\Re$-module $E_d(a,b,c)^{(-1,1)}/E_d(a,b,c)^{(-1,1)}(b)$ as scalar multiplication by 
\begin{gather}\label{delta_V/V(b)}
\frac{(d-1)(d+3)}{16}+\frac{(a-1)(a+1)}{4}
+\frac{b(b+2)}{4}+\frac{(c-1)(c+1)}{4}.
\end{gather}
\end{lem}
\begin{proof}
By Lemma \ref{lem3:iota_Ed(-1,1)} the cosets (\ref{e:basis_V/V(b)}) are an $\F$-basis for $E_d(a,b,c)^{(-1,1)}/E_d(a,b,c)^{(-1,1)}(b)$. 
Applying Lemma \ref{lem:AB_Ed(-1,1)} we obtain the matrices representing $A$ and $B$ with respect to (\ref{e:basis_V/V(b)}). 
By Lemma \ref{lem2:iota_Ed(-1,1)} the element $t_0$ acts on $E_d(a,b,c)^{(-1,1)}/E_d(a,b,c)^{(-1,1)}(b)$ as scalar multiplication by $-b$. Combined with Theorem \ref{thm:hom} and Lemma \ref{lem:t^2_Ed}, the element $\delta$ acts on $E_d(a,b,c)^{(-1,1)}/E_d(a,b,c)^{(-1,1)}(b)$ as scalar multiplication by (\ref{delta_V/V(b)}). The lemma follows.
\end{proof}

\begin{prop}\label{prop:Rmodule_V/V(b)}
Assume that the $\H$-module $E_d(a,b,c)^{(-1,1)}$ is irreducible. 
The $\Re$-module $E_d(a,b,c)^{(-1,1)}/E_d(a,b,c)^{(-1,1)}(b)$ is isomorphic to $$
R_{\frac{d-1}{2}}\left(
-\frac{a+1}{2},
-\frac{b}{2}-1,
-\frac{c+1}{2}
\right).
$$
Moreover the $\Re$-module $E_d(a,b,c)^{(-1,1)}/E_d(a,b,c)^{(-1,1)}(b)$ is irreducible.
\end{prop}
\begin{proof}
Let $(a',b',c',d')=(
-\frac{a+1}{2},
-\frac{b}{2}-1,
-\frac{c+1}{2},
\frac{d-1}{2})$. 
Comparing Proposition \ref{prop:Rd} with Lemma \ref{lem:AB_V/V(b)} yields that the quotient $\Re$-module $E_d(a,b,c)^{(-1,1)}/E_d(a,b,c)^{(-1,1)}(b)$ is isomorphic to $R_{d'}(a',b',c')$. 
It follows from Proposition \ref{prop:irr_E} that 
$$
a'+b'+c'+1',-a'+b'+c',a'+b'-c'\not\in
\left\{
\frac{d'}{2}-i\,\bigg|\, i=1,2,\ldots,d'+1
\right\}
$$
and 
$$
a'-b'+c'\not\in
\left\{
\frac{d'}{2}-i\,\bigg|\, i=0,1,\ldots,d'
\right\}.
$$
By Proposition \ref{prop:irr_R} the $\Re$-module $R_{d'}(a',b',c')$ is irreducible. The proposition follows.
\end{proof}

\begin{thm}\label{thm:Ed(-1,1)}
Assume that the $\H$-module $E_d(a,b,c)^{(-1,1)}$ is irreducible. Then the following hold:
\begin{enumerate}
\item If $b=0$ then 
\begin{table}[H]
\begin{tikzpicture}[node distance=1.2cm]
 \node (0)                  {$\{0\}$};
 \node (E)  [above of=0]   {$E_d(a,b,c)^{(-1,1)}(0)$};
 \node (V)  [above of=E]   {$E_d(a,b,c)^{(-1,1)}$};
 \draw (0)   -- (E);
 \draw (E)  -- (V);
\end{tikzpicture}
\end{table}
\noindent is the lattice of $\Re$-submodules of $E_d(a,b,c)^{(-1,1)}$.

\item If $b\not=0$ then 
 \begin{table}[H]
\begin{tikzpicture}[node distance=1.2cm]
 \node (0)                  {$\{0\}$};
 \node (1)  [above of=0]   {};
 \node (2)  [right of=1]   {};
 \node (3)  [left of=1]   {};
 \node (E1)  [right of=2]  {$E_d(a,b,c)^{(-1,1)}(b)$};
 \node (E2)  [left of=3]   {$E_d(a,b,c)^{(-1,1)}(-b)$};
 \node (V) [above of=1]  {$E_d(a,b,c)^{(-1,1)}$};
 \draw (0)   -- (E1);
 \draw (0)   -- (E2);
 \draw (E1)   -- (V);
 \draw (E2)  -- (V);
\end{tikzpicture}
\end{table}
\noindent is the lattice of $\Re$-submodules of $E_d(a,b,c)^{(-1,1)}$.
\end{enumerate}
\end{thm}
\begin{proof}
Using the above lemmas and propositions, the result follows by an  argument similar to the proof of Theorem \ref{thm:Ed(1,-1)}.
\end{proof}

\subsection{The lattice of $\Re$-submodules of $E_d(a,b,c)^{(-1,-1)}$}\label{s:lattice_Ed(-1,-1)}

Set the parameter 
$$
\tau=a+b-c-\frac{d+1}{2}
$$
as in Proposition \ref{prop:Ed}.

\begin{lem}\label{lem2:iota_Ed(-1,-1)}
The matrix representing $t_0$ with respect to the $\F$-basis 
\begin{gather*}
\textstyle
v_i+(\tau+i)v_{i-1}
\quad 
\hbox{for $i=1,3,\ldots,d$},
\quad 
v_i 
\quad 
\hbox{for $i=0,2,\ldots,d-1$}
\end{gather*}
for $E_d(a,b,c)^{(-1,-1)}$
is 
\begin{gather*}
\begin{pmatrix}
c I_{\frac{d+1}{2}} & \rvline & -I_{\frac{d+1}{2}}
\\
\hline
{\bf 0} & \rvline &-c  I_{\frac{d+1}{2}} 
\end{pmatrix}.
\end{gather*}
\end{lem}
\begin{proof}
By Table \ref{pm1-action} the action of $t_0$ on $E_d(a,b,c)^{(-1,-1)}$ corresponds to the action of $t_1^\vee$ on $E_d(a,b,c)$. 
Applying (\ref{t1vee:Ed}) it is routine to verify the lemma.
\end{proof}

\begin{lem}\label{lem:iota_Ed(-1,-1)}
\begin{enumerate}
\item If $c=0$ then $t_0$ is not diagonalizable on $E_d(a,b,c)^{(-1,-1)}$ with exactly one eigenvalue $0$.

\item If $c\not=0$ then $t_0$ is diagonalizable on $E_d(a,b,c)^{(-1,-1)}$ with exactly two eigenvalues $\pm c$. 
\end{enumerate}
\end{lem}
\begin{proof}
Immediate from Lemma \ref{lem2:iota_Ed(-1,-1)}.
\end{proof}

\begin{lem}\label{lem3:iota_Ed(-1,-1)}
$E_d(a,b,c)^{(-1,-1)}(c)$ is of dimension $\frac{d+1}{2}$ with the $\F$-basis 
\begin{gather*}
\textstyle
v_i+(\tau+i) v_{i-1}
\qquad 
\hbox{for $i=1,3,\ldots,d$}.
\end{gather*}
\end{lem}
\begin{proof}
Immediate from Lemma \ref{lem2:iota_Ed(-1,-1)}.
\end{proof}

\begin{lem}\label{lem:AB_Ed(-1,-1)}
The actions of $A$ and $B$ on the $\H$-module $E_d(a,b,c)^{(-1,-1)}$ are as follows:
\begin{align*}
A v_i &=
\left\{
\begin{array}{ll}
\displaystyle
\theta_i v_{i}+\frac{1}{4} v_{i+2} 
\qquad 
&\hbox{for $i=0,2,\ldots,d-3$},
\\
\displaystyle
\theta_i v_i+\frac{1}{2}v_{i+1}+\frac{1}{4} v_{i+2}
\qquad 
&\hbox{for $i=1,3,\ldots,d-2$},
\end{array}
\right.
\\
A v_{d-1}&=\theta_{d-1} v_{d-1},
\qquad 
A v_d=\theta_d v_d,
\\
B v_i &=
\left\{
\begin{array}{ll}
\displaystyle
\theta^*_i v_i +\frac{i(d-i+1)}{2} v_{i-1}+\frac{i(d-i+1)}{4}\rho_{i-1} v_{i-2}
\qquad 
&\hbox{for $i=2,4,\ldots,d-1$},
\\
\displaystyle
\theta^*_i v_i+\frac{(i-1)(d-i+2)}{4} \rho_i v_{i-2}
\qquad 
&\hbox{for $i=3,5,\ldots, d$},
\end{array}
\right.
\\
B v_0 &= \theta^*_0 v_0,
\qquad 
B v_1 = \theta^*_1 v_1,
\end{align*}
where 
\begin{align*}
\theta_i
&=
\left(\frac{a}{2}-\frac{d-3}{4}+\left\lfloor \frac{i}{2}\right\rfloor\right)
\left(\frac{a}{2}-\frac{d+1}{4}+\left\lfloor \frac{i}{2}\right\rfloor\right)
\qquad (0\leq i\leq d),
\\
\theta_i^*
&=
\left(\frac{b}{2}-\frac{d-3}{4}+\left\lfloor \frac{i}{2}\right\rfloor\right)
\left(\frac{b}{2}-\frac{d+1}{4}+\left\lfloor \frac{i}{2}\right\rfloor\right)
\qquad (0\leq i\leq d).
\end{align*}
\end{lem}
\begin{proof}
By Theorem \ref{thm:hom} and Table \ref{pm1-action} the actions of $A$ and $B$ on $E_d(a,b,c)^{(-1,-1)}$ correspond to the actions of 
$$
\frac{(t_0+t_1)(t_0+t_1+2)}{4},
\qquad 
\frac{(t_0+t_0^\vee)(t_0+t_0^\vee+2)}{4}
$$
on $E_d(a,b,c)$, respectively.
Using Proposition \ref{prop:Ed} it is routine to verify the lemma.
\end{proof}

\begin{lem}\label{lem:AB_V(c)}
The matrices representing $A$ and $B$ with respect to the $\F$-basis 
\begin{gather}\label{basis:V(c)}
\frac{1}{2^{i-1}}(v_i+(\tau+i)v_{i-1})
\qquad \hbox{for $i=1,3,\ldots,d$}
\end{gather}
for the $\Re$-module $E_d(a,b,c)^{(-1,-1)}(c)$ are 
$$
\begin{pmatrix}
\theta_0 & & &  &{\bf 0}
\\ 
1 &\theta_1 
\\
&1 &\theta_2 
 \\
& &\ddots &\ddots
 \\
{\bf 0} & & &1 &\theta_\frac{d-1}{2}
\end{pmatrix},
\qquad 
\begin{pmatrix}
\theta_0^* &\varphi_1 &  & &{\bf 0}
\\ 
 &\theta_1^* &\varphi_2
\\
 &  &\theta_2^* &\ddots
 \\
 & & &\ddots &\varphi_{\frac{d-1}{2}}
 \\
{\bf 0}  & & & &\theta_\frac{d-1}{2}^*
\end{pmatrix}
$$
respectively, where 
\begin{align*}
\theta_i &=
\frac{(2a-d+4i-1)(2a-d+4i+3)}{16}
\qquad (0\leq i\leq \textstyle\frac{d-1}{2}),
\\
\theta_i^* &= 
\frac{(2b-d+4i-1)(2b-d+4i+3)}{16}
\qquad (0\leq i\leq \textstyle\frac{d-1}{2}),
\\
\varphi_i &=
\frac{i(2i-d-1)(2a+2b+2c-d+4i-3)(2a+2b-2c-d+4i+1)}{32}
\qquad 
(1\leq i\leq \textstyle\frac{d-1}{2}).
\end{align*}
The element $\delta$ acts on the $\Re$-module $E_d(a,b,c)^{(-1,-1)}(c)$ as scalar multiplication by 
\begin{gather}\label{delta_V(c)}
\frac{(d-1)(d+3)}{16}+\frac{(a-1)(a+1)}{4}+\frac{(b-1)(b+1)}{4}+\frac{c(c-2)}{4}.
\end{gather}
\end{lem}
\begin{proof}
By Lemma \ref{lem3:iota_Ed(-1,-1)} the vectors (\ref{basis:V(c)}) are an $\F$-basis for $E_d(a,b,c)^{(-1,-1)}(c)$. 
Applying Lemma \ref{lem:AB_Ed(-1,-1)} a straightforward calculation yields the matrices representing $A$ and $B$ with respect to (\ref{basis:V(c)}). Using Theorem \ref{thm:hom} and Lemma \ref{lem:t^2_Ed} yields that $\delta$ acts on $E_d(a,b,c)^{(-1,-1)}(c)$ as scalar multiplication by (\ref{delta_V(c)}). The lemma follows.
\end{proof}

\begin{prop}\label{prop:Rmodule_V(c)}
The $\Re$-module $E_d(a,b,c)^{(-1,-1)}(c)$ is isomorphic to $$
R_{\frac{d-1}{2}}\left(
-\frac{a+1}{2},
-\frac{b+1}{2},
-\frac{c}{2}
\right).
$$
Moreover the $\Re$-module $E_d(a,b,c)^{(-1,-1)}(c)$ is irreducible if the $\H$-module $E_d(a,b,c)^{(-1,-1)}$ is irreducible.
\end{prop}
\begin{proof}
Set $(a',b',c',d')=(
-\frac{a+1}{2},
-\frac{b+1}{2},
-\frac{c}{2},
\frac{d-1}{2})$.
Comparing Proposition \ref{prop:Rd} with Lemma \ref{lem:AB_V(c)} it follows that the $\Re$-module $E_d(a,b,c)^{(-1,-1)}(c)$ is isomorphic to $R_{d'}(a',b',c')$. Suppose that the $\H$-module $E_d(a,b,c)^{(-1,-1)}$ is irreducible. Using Proposition \ref{prop:irr_E} yields that 
$$
a'+b'+c'+1,
-a'+b'+c',
a'-b'+c'
\not\in
\left\{
\frac{d'}{2}-i
\,\bigg|
\,
i=0,1,\ldots,d'
\right\}
$$
and 
\begin{gather*}
a'+b'-c'
\not\in
\left\{
\frac{d'}{2}-i
\,\bigg|
\,
i=1,2,\ldots,d'+1
\right\}.
\end{gather*}
By Proposition \ref{prop:irr_R} the $\Re$-module $R_{d'}(a',b',c')$ is irreducible. The proposition follows.
\end{proof}

\begin{lem}\label{lem:AB_V/V(c)}
The matrices representing $A$ and $B$ with respect to the $\F$-basis 
\begin{gather}\label{e:basis_V/V(c)}
\frac{1}{2^i}v_i+E_d(a,b,c)^{(-1,-1)}(c)
\qquad \hbox{for $i=0,2,\ldots,d-1$}
\end{gather}
for the $\Re$-module $E_d(a,b,c)^{(-1,-1)}/E_d(a,b,c)^{(-1,-1)}(c)$ are 
$$
\begin{pmatrix}
\theta_0 & & &  &{\bf 0}
\\ 
1 &\theta_1 
\\
&1 &\theta_2
 \\
& &\ddots &\ddots
 \\
{\bf 0} & & &1 &\theta_\frac{d-1}{2}
\end{pmatrix},
\qquad 
\begin{pmatrix}
\theta_0^* &\varphi_1 &  & &{\bf 0}
\\ 
 &\theta_1^* &\varphi_2
\\
 &  &\theta_2^* &\ddots
 \\
 & & &\ddots &\varphi_{\frac{d-1}{2}}
 \\
{\bf 0}  & & & &\theta_\frac{d-1}{2}^*
\end{pmatrix}
$$
respectively, where 
\begin{align*}
\theta_i &=
\frac{(2a-d+4i-1)(2a-d+4i+3)}{16}
\qquad (0\leq i\leq \textstyle\frac{d-1}{2}),
\\
\theta_i^* &= 
\frac{(2b-d+4i-1)(2b-d+4i+3)}{16}
\qquad (0\leq i\leq \textstyle\frac{d-1}{2}),
\\
\varphi_i &=
\frac{i(2i-d-1)(2a+2b+2c-d+4i+1)(2a+2b-2c-d+4i-3)}{32}
\qquad 
(1\leq i\leq \textstyle\frac{d-1}{2}).
\end{align*}
The element $\delta$ acts on the $\Re$-module $E_d(a,b,c)^{(-1,-1)}/E_d(a,b,c)^{(-1,-1)}(c)$ as scalar multiplication by 
\begin{gather}\label{delta_V/V(c)}
\frac{(d-1)(d+3)}{16}+\frac{(a-1)(a+1)}{4}+\frac{(b-1)(b+1)}{4}+\frac{c(c+2)}{4}.
\end{gather}
\end{lem}
\begin{proof}
By Lemma \ref{lem3:iota_Ed(-1,-1)} the cosets (\ref{e:basis_V/V(c)}) are an $\F$-basis for $E_d(a,b,c)^{(-1,-1)}/E_d(a,b,c)^{(-1,-1)}(c)$. 
Applying Lemma \ref{lem:AB_Ed(-1,-1)} a direct calculation yields the matrices representing $A$ and $B$ with respect to (\ref{e:basis_V/V(c)}). 
By Lemma \ref{lem2:iota_Ed(-1,-1)} the element $t_0$ acts on $E_d(a,b,c)^{(-1,-1)}/E_d(a,b,c)^{(-1,-1)}(c)$ as scalar multiplication by $-c$. Combined with Theorem \ref{thm:hom} and Lemma \ref{lem:t^2_Ed} the element $\delta$ acts on $E_d(a,b,c)^{(-1,-1)}/E_d(a,b,c)^{(-1,-1)}(c)$ as scalar multiplication by (\ref{delta_V/V(c)}). The lemma follows.
\end{proof}

\begin{prop}\label{prop:Rmodule_V/V(c)}
The $\Re$-module $E_d(a,b,c)^{(-1,-1)}/E_d(a,b,c)^{(-1,-1)}(c)$ is isomorphic to $$
R_{\frac{d-1}{2}}\left(
-\frac{a+1}{2},
-\frac{b+1}{2},
-\frac{c}{2}-1
\right).
$$
Moreover the $\Re$-module 
$
E_d(a,b,c)^{(-1,-1)}/ E_d(a,b,c)^{(-1,-1)}(c)
$ 
is irreducible provided that the $\H$-module 
 $E_d(a,b,c)^{(-1,-1)}$ is irreducible.
\end{prop}
\begin{proof}
Let $(a',b',c',d')=(
-\frac{a+1}{2},
-\frac{b+1}{2},
-\frac{c}{2}-1,
\frac{d-1}{2})$. 
Comparing Proposition \ref{prop:Rd} with Lemma \ref{lem:AB_V/V(c)} yields that the quotient $\Re$-module $E_d(a,b,c)^{(-1,-1)}/E_d(a,b,c)^{(-1,-1)}(c)$ is isomorphic to $R_{d'}(a',b',c')$. Suppose that the $\H$-module $E_d(a,b,c)^{(-1,-1)}$ is irreducible.
By Proposition \ref{prop:irr_E} we have 
$$
a'+b'+c'+1',-a'+b'+c',a'-b'+c'\not\in
\left\{
\frac{d'}{2}-i\,\bigg|\, i=1,2,\ldots,d'+1
\right\}
$$
and 
$$
a'+b'-c'
\not\in
\left\{
\frac{d'}{2}-i\,\bigg|\, i=0,1,\ldots,d'
\right\}.
$$
Combined with Proposition \ref{prop:irr_R} the $\Re$-module $R_{d'}(a',b',c')$ is irreducible. The proposition follows.
\end{proof}

\begin{thm}\label{thm:Ed(-1,-1)}
Assume that the $\H$-module $E_d(a,b,c)^{(-1,-1)}$ is irreducible. Then the following hold:
\begin{enumerate}
\item If $c=0$ then 
\begin{table}[H]
\begin{tikzpicture}[node distance=1.2cm]
 \node (0)                  {$\{0\}$};
 \node (E)  [above of=0]   {$E_d(a,b,c)^{(-1,-1)}(0)$};
 \node (V)  [above of=E]   {$E_d(a,b,c)^{(-1,-1)}$};
 \draw (0)   -- (E);
 \draw (E)  -- (V);
\end{tikzpicture}
\end{table}
\noindent is the lattice of $\Re$-submodules of $E_d(a,b,c)^{(-1,-1)}$.

\item If $c\not=0$ then 
 \begin{table}[H]
\begin{tikzpicture}[node distance=1.2cm]
 \node (0)                  {$\{0\}$};
 \node (1)  [above of=0]   {};
 \node (2)  [right of=1]   {};
 \node (3)  [left of=1]   {};
 \node (E1)  [right of=2]  {$E_d(a,b,c)^{(-1,-1)}(c)$};
 \node (E2)  [left of=3]   {$E_d(a,b,c)^{(-1,-1)}(-c)$};
 \node (V) [above of=1]  {$E_d(a,b,c)^{(-1,-1)}$};
 \draw (0)   -- (E1);
 \draw (0)   -- (E2);
 \draw (E1)   -- (V);
 \draw (E2)  -- (V);
\end{tikzpicture}
\end{table}
\noindent is the lattice of $\Re$-submodules of $E_d(a,b,c)^{(-1,-1)}$.
\end{enumerate}
\end{thm}
\begin{proof}
Using the above lemmas and propositions, the result follows by an argument similar to the proof of Theorem \ref{thm:Ed(1,-1)}.
\end{proof}

\subsection{The lattice of $\Re$-submodules of $O_d(a,b,c)$}\label{s:lattice_Od}

Throughout this subsection we 
let $d\geq 0$ denote an even integer and let $\{v_i\}_{i=0}^d$ denote the $\F$-basis for $O_d(a,b,c)$ from Proposition \ref{prop:Od}. For notational convenience we set the parameters 
$$
\sigma=a+b+c-\frac{d+1}{2},
\qquad 
\tau=a+b-c-\frac{d+1}{2}
$$
as in Proposition \ref{prop:Od}.

\begin{lem}\label{lem2:iota_Od}
The matrix representing $t_0$ with respect to the $\F$-basis 
\begin{gather*}
v_0,
\quad 
v_i-i v_{i-1}
\quad 
\hbox{for $i=2,4,\ldots,d$},
\quad 
v_i 
\quad 
\hbox{for $i=1,3,\ldots,d-1$}
\end{gather*}
for $O_d(a,b,c)$ is 
\begin{gather*}
\begin{pmatrix}
\textstyle
\frac{\sigma}{2}
&\rvline &{\bf 0} &\rvline &{\bf 0}
\\
\hline
{\bf 0} &\rvline &
\frac{\sigma}{2}
I_{\frac{d}{2}} 
& \rvline &I_{\frac{d}{2}}
\\
\hline
{\bf 0} & \rvline &{\bf 0} & \rvline 
&
-\frac{\sigma}{2}
I_{\frac{d}{2}} 
\end{pmatrix}.
\end{gather*} 
\end{lem}
\begin{proof}
It is straightforward to verify the lemma by using Proposition \ref{prop:Od}.
\end{proof}

\begin{lem}\label{lem:iota_Od}
\begin{enumerate}
\item If $d=0$ then $t_0$ is diagonalizable on $O_d(a,b,c)$ with exactly one eigenvalue $\frac{\sigma}{2}$.  

\item If $d\geq 2$ and $a+b+c=\frac{d+1}{2}$ then $t_0$ is not diagonalizable on $O_d(a,b,c)$ with exactly one eigenvalue $0$.

\item If $d\geq 2$ and $a+b+c\not=\frac{d+1}{2}$ then $t_0$ is diagonalizable on $O_d(a,b,c)$ with exactly two eigenvalues 
$\pm\frac{\sigma}{2}$. 
\end{enumerate}
\end{lem}
\begin{proof}
Immediate from Lemma \ref{lem2:iota_Od}.
\end{proof}

\begin{lem}\label{lem3:iota_Od}
$O_d(a,b,c)(
\frac{\sigma}{2})$ is of dimension $\frac{d}{2}+1$ with the $\F$-basis
$$
v_0,
\quad
v_i-i v_{i-1}
\quad
\hbox{for $i=2,4,\ldots,d$}.
$$
\end{lem}
\begin{proof}
Immediate from Lemma \ref{lem2:iota_Od}.
\end{proof}

\begin{lem}\label{lem:AB_Od}
The actions of $A$ and $B$ on the $\H$-module $O_d(a,b,c)$ are as follows:
\begin{align*}
A v_i &=
\left\{
\begin{array}{ll}
\theta_{i}
\displaystyle v_{i}-
\frac{1}{2} v_{i+1}+\frac{1}{4} v_{i+2}
\qquad 
&\hbox{for $i=0,2,\ldots,d-2$},
\\
\displaystyle \theta_i v_i
+\frac{1}{4} v_{i+2}
\qquad 
&\hbox{for $i=1,3,\ldots,d-3$},
\end{array}
\right.
\\
A v_{d-1}&=\theta_{d-1} v_{d-1},
\qquad 
A v_d =\theta_d v_{d},
\\
B v_i &=
\left\{
\begin{array}{ll}
\displaystyle\theta_i^* v_{i}+
 \frac{i(i-d-2)(\sigma+i)(\tau+{i-1})}{4} v_{i-2}
\qquad 
&\hbox{for $i=2,4,\ldots,d$},
\\
\displaystyle\theta^*_i v_i
+\frac{(i-d-1)(\tau+i)}{2}\left(v_{i-1}+\frac{(i-1)(\sigma+{i-1})}{2} v_{i-2}\right)
\qquad 
&\hbox{for $i=3,5,\ldots,d-1$},
\end{array}
\right.
\\
B v_0 &=\theta^*_0 v_0,
\qquad 
B v_1 =\theta^*_1 v_1-\frac{d(\tau+1)}{2} v_0.
\end{align*}
where 
\begin{align*}
\theta_i
&=
\left(\frac{a}{2}-\frac{d+3}{4}+\left\lceil \frac{i}{2}\right\rceil\right)
\left(\frac{a}{2}-\frac{d-1}{4}+\left\lceil \frac{i}{2}\right\rceil\right)
\qquad (0\leq i\leq d),
\\
\theta_i^*
&=
\left(\frac{b}{2}-\frac{d+3}{4}+\left\lceil \frac{i}{2}\right\rceil\right)
\left(\frac{b}{2}-\frac{d-1}{4}+\left\lceil \frac{i}{2}\right\rceil\right)
\qquad (0\leq i\leq d).
\end{align*}
\end{lem}
\begin{proof}
Apply Theorem \ref{thm:hom} and Proposition \ref{prop:Od} to evaluate the actions of $A$ and $B$ on $O_d(a,b,c)$.
\end{proof}

\begin{lem}\label{lem:AB_V(a+b+c-d/2)}
The matrices representing $A$ and $B$ with respect to the $\F$-basis 
\begin{gather}\label{e:basis_V(a+b+c-d/2)}
v_0, \quad 
\frac{1}{2^i}(v_i-i v_{i-1})
\quad 
\hbox{for $i=2,4,\ldots,d$}
\end{gather}
for the $\Re$-module $O_d(a,b,c)(\frac{\sigma}{2})$ are 
$$
\begin{pmatrix}
\theta_0 & & &  &{\bf 0}
\\ 
1 &\theta_1 
\\
&1 &\theta_2
 \\
& &\ddots &\ddots
 \\
{\bf 0} & & &1 &\theta_\frac{d}{2}
\end{pmatrix},
\qquad 
\begin{pmatrix}
\theta_0^* &\varphi_1 &  & &{\bf 0}
\\ 
 &\theta_1^* &\varphi_2
\\
 &  &\theta_2^* &\ddots
 \\
 & & &\ddots &\varphi_{\frac{d}{2}}
 \\
{\bf 0}  & & & &\theta^*_\frac{d}{2}
\end{pmatrix}
$$
respectively, where 
\begin{align*}
\theta_i
&=\frac{(2a-d+4i-3)(2a-d+4i+1)}{16} 
\qquad (0\leq i\leq \textstyle\frac{d}{2}),
\\
\theta_i^*
&=\frac{(2b-d+4i-3)(2b-d+4i+1)}{16}
\qquad (0\leq i\leq \textstyle\frac{d}{2}),
\\
\varphi_i &=
\frac{i(2i-d-2)(2a+2b+2c-d+4i-5)(2a+2b-2c-d+4i-3)}{32}
\qquad 
(1\leq i\leq \textstyle\frac{d}{2}).
\end{align*}
The element $\delta$ acts on the $\Re$-module $O_d(a,b,c)(\frac{\sigma}{2})$ as scalar multiplication by 
\begin{gather}\label{delta_V(a+b+c-d/2)}
\frac{d(d+4)}{16}+\frac{(2a-3)(2a+1)}{16}+\frac{(2b-3)(2b+1)}{16}+\frac{(2c-3)(2c+1)}{16}.
\end{gather}
\end{lem}
\begin{proof}
By Lemma \ref{lem3:iota_Od} the vectors (\ref{e:basis_V(a+b+c-d/2)}) are an $\F$-basis for $O_d(a,b,c)(\frac{\sigma}{2})$. 
Applying Lemma \ref{lem:AB_Od} a straightforward calculation yields the matrices representing $A$ and $B$ with respect to (\ref{e:basis_V(a+b+c-d/2)}). Applying Theorem \ref{thm:hom} and Lemma \ref{lem:t^2_Od} yields that $\delta$ acts on $O_d(a,b,c)(\frac{\sigma}{2})$ as scalar multiplication by (\ref{delta_V(a+b+c-d/2)}). The lemma follows.
\end{proof}

\begin{prop}\label{prop:Rmodule_V(a+b+c-d/2)}
The $\Re$-module $O_d(a,b,c)(\frac{\sigma}{2})$ is isomorphic to 
$$
R_{\frac{d}{2}}\left(
-\frac{a}{2}-\frac{1}{4},
-\frac{b}{2}-\frac{1}{4},
-\frac{c}{2}-\frac{1}{4}
\right).
$$
Moreover the $\Re$-module $O_d(a,b,c)(\frac{\sigma}{2})$ is irreducible provided that $a+b+c\not=\frac{d+1}{2}$ and the $\H$-module $O_d(a,b,c)$ is irreducible.
\end{prop}
\begin{proof}
Set $(a',b',c',d')=(
-\frac{a}{2}-\frac{1}{4},
-\frac{b}{2}-\frac{1}{4},
-\frac{c}{2}-\frac{1}{4},
\frac{d}{2})$.
Comparing Proposition \ref{prop:Rd} with Lemma \ref{lem:AB_V(a+b+c-d/2)} yields that the $\Re$-module $O_d(a,b,c)(\frac{\sigma}{2})$ is isomorphic to $R_{d'}(a',b',c')$. Suppose that  $a+b+c\not=\frac{d+1}{2}$ and 
the $\H$-module $O_d(a,b,c)$ is irreducible. It follows from Proposition \ref{prop:irr_O} that 
\begin{gather*}\label{a+b+c+1_V(a+b+c-d/2)}
a'+b'+c'+1\not\in 
\left\{
\frac{d'}{2}-i
\,\bigg|
\,
i=0,1,\ldots,d'-1
\right\}
\end{gather*}
and 
\begin{gather*}
-a'+b'+c',
a'-b'+c',
a'+b'-c'
\not\in
\left\{
\frac{d'}{2}-i
\,\bigg|
\,
i=1,2,\ldots,d'
\right\}.
\end{gather*}
By the assumption $a+b+c\not=\frac{d+1}{2}$ we have 
$
a'+b'+c'+1\not=
-\frac{d'}{2}$. 
By Proposition \ref{prop:irr_R} the $\Re$-module $R_{d'}(a',b',c')$ is irreducible. The proposition follows.
\end{proof}

\begin{lem}\label{lem:AB_V/V(a+b+c-d/2)}
Assume that $d\geq 2$. The matrices representing $A$ and $B$ with respect to the $\F$-basis 
\begin{gather}\label{e:basis_V/V(a+b+c-d/2)}
\frac{1}{2^{i-1}}v_i+O_d(a,b,c)(
\textstyle
\frac{\sigma}{2}
)
\qquad 
\hbox{for $i=1,3,\ldots,d-1$}
\end{gather}
for the $\Re$-module $O_d(a,b,c)/O_d(a,b,c)(\frac{\sigma}{2})$ are 
$$
\begin{pmatrix}
\theta_0 & & &  &{\bf 0}
\\ 
1 &\theta_1 
\\
&1 &\theta_2
 \\
& &\ddots &\ddots
 \\
{\bf 0} & & &1 &\theta_{\frac{d}{2}-1}
\end{pmatrix},
\qquad 
\begin{pmatrix}
\theta^*_0 &\varphi_1 &  & &{\bf 0}
\\ 
 &\theta^*_1 &\varphi_2
\\
 &  &\theta_2^* &\ddots
 \\
 & & &\ddots &\varphi_{\frac{d}{2}-1}
 \\
{\bf 0}  & & & &\theta^*_{\frac{d}{2}-1}
\end{pmatrix}
$$
respectively, where 
\begin{align*}
\theta_i &= \frac{(2a-d+4i+1)(2a-d+4i+5)}{16}
\qquad 
(0\leq i\leq \textstyle\frac{d}{2}-1),
\\
\theta^*_i &=
\frac{(2b-d+4i+1)(2b-d+4i+5)}{16}
\qquad 
(0\leq i\leq \textstyle\frac{d}{2}-1),
\\
\varphi_i &=\frac{i(2i-d)(2a+2b+2c-d+4i+3)(2a+2b-2c-d+4i+1)}{32}
\qquad (1\leq i\leq \textstyle\frac{d}{2}-1).
\end{align*}
The element $\delta$ acts on $O_d(a,b,c)/O_d(a,b,c)(\frac{\sigma}{2})$ as scalar multiplication by 
\begin{gather}\label{delta_V/V(a+b+c-d/2)}
\frac{d^2-13}{16}+\frac{a(a+1)}{4}+\frac{b(b+1)}{4}+\frac{c(c+1)}{4}.
\end{gather}
\end{lem}
\begin{proof}
By Lemma \ref{lem3:iota_Od} the cosets (\ref{e:basis_V/V(a+b+c-d/2)}) are an $\F$-basis for $O_d(a,b,c)/O_d(a,b,c)(\frac{\sigma}{2})$. 
Applying Lemma \ref{lem:AB_Od} a direct calculation yields the matrices representing $A$ and $B$ with respect to (\ref{e:basis_V/V(a+b+c-d/2)}). By Theorem \ref{thm:hom} and Lemma \ref{lem:t^2_Od} the element $\delta$ acts on $O_d(a,b,c)/O_d(a,b,c)(\frac{\sigma}{2})$ as scalar multiplication by (\ref{delta_V/V(a+b+c-d/2)}). The lemma follows.
\end{proof}

\begin{prop}\label{prop:Rmodule_V/V(a+b+c-d/2)}
Assume that $d\geq 2$. Then the $\Re$-module $O_d(a,b,c)/O_d(a,b,c)(\frac{\sigma}{2})$ is isomorphic to 
$$
R_{\frac{d}{2}-1}\left(
-\frac{a}{2}-\frac{3}{4},
-\frac{b}{2}-\frac{3}{4},
-\frac{c}{2}-\frac{3}{4}
\right).
$$
Moreover the $\Re$-module 
$
O_d(a,b,c)/O_d(a,b,c)(\frac{\sigma}{2})
$ 
is irreducible provided that the $\H$-module $O_d(a,b,c)$ is irreducible.
\end{prop}
\begin{proof}
Set $(a',b',c',d')=(
-\frac{a}{2}-\frac{3}{4},
-\frac{b}{2}-\frac{3}{4},
-\frac{c}{2}-\frac{3}{4},
\frac{d}{2}-1)$.
Comparing Proposition \ref{prop:Rd} with Lemma \ref{lem:AB_V/V(a+b+c-d/2)} it follows that the $\Re$-module $O_d(a,b,c)/O_d(a,b,c)(\frac{\sigma}{2})$ is isomorphic to $R_{d'}(a',b',c')$. Suppose that the $\H$-module $O_d(a,b,c)$ is irreducible. Using Proposition \ref{prop:irr_O} yields that 
\begin{gather*}
a'+b'+c'+1,
-a'+b'+c',
a'-b'+c',
a'+b'-c'
\not\in
\left\{
\frac{d'}{2}-i
\,\bigg|
\,
i=1,2,\ldots,d'+1
\right\}.
\end{gather*}
By Proposition \ref{prop:irr_R} the $\Re$-module $R_{d'}(a',b',c')$ is irreducible. The proposition follows.
\end{proof}

For the rest of this subsection we let $O_d(a,b,c)(0)'$ denote the $\F$-subspace of 
$O_d(a,b,c)(0)$ spanned by
$$
v_i-i v_{i-1}
\qquad 
\hbox{for all $i=2,4,\ldots,d$}.
$$

\begin{lem}\label{lem:Rmodule_V(0)'}
Assume that $d\geq 2$ and $a+b+c=\frac{d+1}{2}$. Then  $O_d(a,b,c)(0)'$ 
is an $\Re$-module and the actions of $A,B,\delta$ on $O_d(a,b,c)(0)'$ are as follows: The matrices representing $A$ and $B$ with respect to the $\F$-basis 
\begin{gather}\label{e:basis_V(0)'}
\frac{1}{2^{i-2}}(v_i-i v_{i-1})
\qquad 
\hbox{for $i=2,4,\ldots,d$}
\end{gather}  
for the $\Re$-module $O_d(a,b,c)(0)'$ are 
$$
\begin{pmatrix}
\theta_0 & & &  &{\bf 0}
\\ 
1 &\theta_1 
\\
&1 &\theta_2
 \\
& &\ddots &\ddots
 \\
{\bf 0} & & &1 &\theta_{\frac{d}{2}-1}
\end{pmatrix},
\qquad 
\begin{pmatrix}
\theta_0^* &\varphi_1 &  & &{\bf 0}
\\ 
 &\theta_1^* &\varphi_2
\\
 &  &\theta_2^* &\ddots
 \\
 & & &\ddots &\varphi_{\frac{d}{2}-1}
 \\
{\bf 0}  & & & &\theta^*_{\frac{d}{2}-1}
\end{pmatrix}
$$
respectively, where 
\begin{align*}
\theta_i
&=\frac{(2a-d+4i+1)(2a-d+4i+5)}{16} 
\qquad (0\leq i\leq \textstyle\frac{d}{2}-1),
\\
\theta_i^*
&=\frac{(2b-d+4i+1)(2b-d+4i+5)}{16}
\qquad (0\leq i\leq \textstyle\frac{d}{2}-1),
\\
\varphi_i &=
\frac{i(2i-d)(2a+2b+2c-d+4i+3)(2a+2b-2c-d+4i+1)}{32}
\qquad 
(1\leq i\leq \textstyle\frac{d}{2}-1).
\end{align*}
The element $\delta$ acts on $O_d(a,b,c)(0)'$ as scalar multiplication by 
\begin{gather}\label{delta_V(0)'}
\frac{d^2-13}{16}+\frac{a(a+1)}{4}+\frac{b(b+1)}{4}+\frac{c(c+1)}{4}.
\end{gather}
\end{lem}
\begin{proof}
It follows from Lemma \ref{lem:AB_V(a+b+c-d/2)} that $O_d(a,b,c)(0)'$ is invariant under $A$ and $\delta$; under the assumption $a+b+c=\frac{d+1}{2}$ it is also invariant under $B$.
Hence $O_d(a,b,c)(0)'$ is an $\Re$-module by Lemma \ref{lem:delta}(ii). 

By Lemma \ref{lem:AB_V(a+b+c-d/2)} the matrix representing $A$ with respect to the $\F$-basis (\ref{e:basis_V(0)'}) for $O_d(a,b,c)(0)'$ is as stated. Under the assumption $a+b+c=\frac{d+1}{2}$ the matrix representing $B$ with respect to (\ref{e:basis_V(0)'}) is as stated and the scalars (\ref{delta_V(a+b+c-d/2)}) and (\ref{delta_V(0)'}) are identical. The lemma follows.
\end{proof}

\begin{prop}\label{prop:Rmodule_V(0)'}
Assume that $d\geq 2$ and $a+b+c=\frac{d+1}{2}$. Then the following hold:
\begin{enumerate}
\item The $\Re$-module $O_d(a,b,c)(0)'$ is isomorphic to 
$$
R_{\frac{d}{2}-1}\left(
-\frac{a}{2}-\frac{3}{4},
-\frac{b}{2}-\frac{3}{4},
-\frac{c}{2}-\frac{3}{4}
\right).
$$

\item If the $\H$-module $O_d(a,b,c)$ is irreducible then the $\Re$-module $O_d(a,b,c)(0)'$ is irreducible.

\item The $\Re$-module $O_d(a,b,c)(0)/O_d(a,b,c)(0)'$ is isomorphic to $R_0(-\frac{b+c+1}{2},-\frac{c+a+1}{2},-\frac{a+b+1}{2})$.
\end{enumerate}
\end{prop}
\begin{proof}
(i), (ii): Comparing Lemmas \ref{lem:AB_V/V(a+b+c-d/2)} and \ref{lem:Rmodule_V(0)'} the statements (i), (ii) are immediate from Proposition \ref{prop:Rmodule_V/V(a+b+c-d/2)}.

(iii): Set $(a',b',c')=(
-\frac{b+c+1}{2},-\frac{c+a+1}{2},-\frac{a+b+1}{2})$. Under the assumption $a+b+c=\frac{d+1}{2}$, it follows from Lemma \ref{lem:AB_V(a+b+c-d/2)} that $A,B,\delta$ act on the $\Re$-module $O_d(a,b,c)(0)/O_d(a,b,c)(0)'$ as scalar multiplication by 
$$
a'(a'+1),
\qquad 
b'(b'+1),
\qquad
a'(a'+1)+b'(b'+1)+c'(c'+1),
$$
respectively. Hence the $\Re$-module $O_d(a,b,c)(0)/O_d(a,b,c)(0)'$ is isomorphic to the $\Re$-module $R_0(a',b',c')$ by Proposition \ref{prop:Rd}.
\end{proof}

\begin{thm}\label{thm:Od}
Assume that the $\H$-module $O_d(a,b,c)$ is irreducible. Then the following hold:
\begin{enumerate}
\item If $d=0$ then the $\Re$-module $O_d(a,b,c)$ is irreducible.
\item If $d\geq 2$ and $a+b+c=\frac{d+1}{2}$ then 
\begin{table}[H]
\begin{tikzpicture}[node distance=1.2cm]
 \node (0)                  {$\{0\}$};
 \node (1)  [above of=0]   {$O_d(a,b,c)(0)'$};
 \node (E)  [above of=1]   {$O_d(a,b,c)(0)$};
 \node (V)  [above of=E]   {$O_d(a,b,c)$};
 \draw (0)   -- (1);
 \draw (1)   -- (E);
 \draw (E)  -- (V);
\end{tikzpicture}
\end{table}
\noindent is the lattice of $\Re$-submodules of $O_d(a,b,c)$.

\item  If $d\geq 2$ and $a+b+c\not=\frac{d+1}{2}$ then 
 \begin{table}[H]
\begin{tikzpicture}[node distance=1.2cm]
 \node (0)                  {$\{0\}$};
 \node (1)  [above of=0]   {};
 \node (2)  [right of=1]   {};
 \node (3)  [left of=1]   {};
 \node (E1)  [right of=2]  {$O_d(a,b,c)(\frac{\sigma}{2})$};
 \node (E2)  [left of=3]   {$O_d(a,b,c)(-\frac{\sigma}{2})$};
 \node (V) [above of=1]  {$O_d(a,b,c)$};
 \draw (0)   -- (E1);
 \draw (0)   -- (E2);
 \draw (E1)   -- (V);
 \draw (E2)  -- (V);
\end{tikzpicture}
\end{table}
\noindent is the lattice of $\Re$-submodules of $O_d(a,b,c)$.
\end{enumerate}
\end{thm}
\begin{proof}
(i): 
If $d=0$ then  $O_d(a,b,c)$ is one-dimensional and hence an irreducible $\Re$-module.

(ii): Suppose that $d\geq 2$ and $a+b+c=\frac{d+1}{2}$. Combined with Propositions \ref{prop:Rmodule_V/V(a+b+c-d/2)} and \ref{prop:Rmodule_V(0)'} the sequence 
\begin{gather}\label{csOd'}
\textstyle
\{0\}\subset O_d(a,b,c)(0)'\subset O_d(a,b,c)(0)\subset O_d(a,b,c)
\end{gather}
is a composition series for the $\Re$-module $O_d(a,b,c)$.

By Proposition \ref{prop:irr_Rmodule_in_t0eigenspace} and Lemma \ref{lem:iota_Od}(ii), every irreducible $\Re$-submodule of $O_d(a,b,c)$ is contained in $O_d(a,b,c)(0)$. 
To see (ii), it remains to show that $O_d(a,b,c)(0)'$ is the unique irreducible $\Re$-submodule of $O_d(a,b,c)(0)$. Suppose on the contrary that $W$ is an irreducible $\Re$-submodule $O_d(a,b,c)(0)$ different from $O_d(a,b,c)(0)'$. By irreducibility, we have  $O_d(a,b,c)(0)'\cap W=\{0\}$. Since $O_d(a,b,c)(0)'$ is of codimension $1$ in $O_d(a,b,c)(0)$, it follows that $W$ is of dimension $1$ and 
\begin{gather}\label{direcsum}
O_d(a,b,c)(0)=O_d(a,b,c)(0)'\oplus W.
\end{gather}
Applying Jordan--H\"{o}lder theorem to (\ref{csOd'}) the one-dimensional $\Re$-module $W$ is isomorphic to $O_d(a,b,c)(0)'$ when $d=2$ or $O_d(a,b,c)(0)/O_d(a,b,c)(0)'$.

First we suppose that $d=2$ and the $\Re$-module $W$ is isomorphic to $O_d(a,b,c)(0)'$. By Lemma \ref{lem:AB_V(a+b+c-d/2)} 
the eigenvalues of $A$ in $O_d(a,b,c)(0)$ are 
\begin{align*}
\theta_0 &=\frac{(2a-5)(2a-1)}{16},
\\
\theta_1 &=\frac{(2a-1)(2a+3)}{16}.
\end{align*}
By Lemma \ref{lem:Rmodule_V(0)'} 
the eigenvalue of $A$ in $O_d(a,b,c)(0)'$ is 
$
\theta_1$.
Combined with (\ref{direcsum}) this implies $\theta_0=\theta_1$.
It follows that 
$
a=\frac{1}{2}.
$ 
By considering the eigenvalues of $B$ in $O_d(a,b,c)(0)$ and $O_d(a,b,c)(0)'$, a similar argument implies $b=\frac{1}{2}$. Moreover $c=\frac{1}{2}$ by the assumption $a+b+c=\frac{d+1}{2}$. Then 
$$
a-b-c=-a+b-c=-a-b+c=-\frac{1}{2}.
$$  
This leads to 
a contradiction to the irreducibility of the $\H$-module $O_d(a,b,c)$ by Proposition \ref{prop:irr_O}.

Next we suppose that $W$ is isomorphic to $O_d(a,b,c)(0)/O_d(a,b,c)(0)'$. By Proposition \ref{prop:Rmodule_V(0)'}(iii) the elements $A$ and $B$ act on $W$ as the scalars 
\begin{align*}
\theta_0  &=\frac{(b+c-1)(b+c+1)}{4}=\frac{(2a-d-3)(2a-d+1)}{16},
\\
\theta_0^*&=\frac{(c+a-1)(c+a+1)}{4}=\frac{(2b-d-3)(2b-d+1)}{16},
\end{align*} 
respectively. By Lemma \ref{lem:AB_V(a+b+c-d/2)} the $\theta_0$-eigenspace of $A$ in $O_d(a,b,c)(0)$ is one-dimensional and hence is equal to $W$. Consequently $W$ contains a vector $w$ in which the coefficient of $\frac{1}{2^d}
(v_d-d v_{d-1})$ with respect to the $\F$-basis (\ref{e:basis_V(a+b+c-d/2)}) for $O_d(a,b,c)(0)$ is $1$. By Lemma \ref{lem:AB_V(a+b+c-d/2)} the coefficient of $\frac{1}{2^d}(v_d-d v_{d-1})$ in $Bw$ with respect to (\ref{e:basis_V(a+b+c-d/2)}) is 
$$
\theta_\frac{d}{2}^*=\frac{(2b+d-3)(2b+d+1)}{16}.
$$ 
Since $w$ is a $\theta_0^*$-eigenvector of $B$ it follows that 
$\theta_0^*=\theta_\frac{d}{2}^*$.
Hence $b=\frac{1}{2}$. Combined with the assumption $a+b+c=\frac{d+1}{2}$ we have 
$$
-a+b-c=\frac{1-d}{2}.
$$ 
This leads to a contradiction to the irreducibility of the $\H$-module $O_d(a,b,c)$ by Proposition \ref{prop:irr_O}. We have shown that $O_d(a,b,c)(0)'$ is the unique irreducible $\Re$-submodule of $O_d(a,b,c)(0)$. Therefore (ii) follows.

(iii): Using the above lemmas and propositions, the statement (iii) follows by an argument similar to the proof of Theorem \ref{thm:Ed}(ii).
\end{proof}

\section{The summary}\label{s:proof}

We summarize the results of \S\ref{s:lattice_Ed}--\S\ref{s:lattice_Od} as follows:

\begin{thm}\label{thm:BIasRmodule}
Let $V$ denote a finite-dimensional irreducible $\H$-module. Given any $\theta\in \F$ let $V(\theta)$ denote the null space of $t_0-\theta$ in $V$. Then the following hold:
\begin{enumerate}
\item Suppose that $t_0$ is not diagonalizable on $V$. Then $0$ is the unique eigenvalue of $t_0$ in $V$. Moreover the following hold:
\begin{enumerate}
\item If the dimension of $V$ is even then the lattice of $\Re$-submodules of $V$ is as follows:
\begin{table}[H]
\begin{tikzpicture}[node distance=1.2cm]
 \node (0)                  {$\{0\}$};
 \node (E)  [above of=0]   {$V(0)$};
 \node (V) [above of=E]  {$V$};
 \draw (0)   -- (E);
 \draw (E)  -- (V);
\end{tikzpicture}
\end{table}

\item If the dimension of $V$ is odd then the lattice of $\Re$-submodules of $V$ is as follows:
\begin{table}[H]
\begin{tikzpicture}[node distance=1.2cm]
 \node (0)                  {$\{0\}$};
 \node (W)  [above of=0]   {$V(0)'$};
 \node (E)  [above of=W]   {$V(0)$};
 \node (V) [above of=E]  {$V$};
 \draw (0)   -- (W);
 \draw (W)   -- (E);
 \draw (E)  -- (V);
\end{tikzpicture}
\end{table}
\noindent Here $V(0)'$ is the irreducible $\Re$-submodule of $V(0)$ that has codimension $1$.
\end{enumerate}

\item Suppose that $t_0$ is diagonalizable on $V$. Then there are at most two eigenvalues of $t_0$ in $V$. Moreover the following hold:
\begin{enumerate}
\item If $t_0$ has exactly one eigenvalue in $V$ then the $\Re$-module $V$ is irreducible of dimension less than or equal to $2$. 

\item If $t_0$ has exactly two eigenvalues in $V$ then there exists a nonzero scalar $\theta\in \F$ such that $\pm \theta$ are the eigenvalues of $t_0$ and the lattice of $\Re$-submodules of $V$ is as follows:
\begin{table}[H]
\begin{tikzpicture}[node distance=1.2cm]
 \node (0)                  {$\{0\}$};
 \node (W)  [above of=0]   {};
 \node (E1)  [left of=W]   {$V(-\theta)$};
 \node (E2) [right of=W]  {$V(\theta)$};
 \node (V)  [above of=W]   {$V$};
 \draw (0)   -- (E1);
 \draw (0)   -- (E2);
 \draw (E1)  -- (V);
 \draw (E2)  -- (V);
\end{tikzpicture}
\end{table}
\end{enumerate}
\end{enumerate}
\end{thm}

As byproducts of Theorem \ref{thm:BIasRmodule} we have the following corollaries:

\begin{cor}
Let $V$ denote a finite-dimensional irreducible $\H$-module. If $\theta$ is an eigenvalue of $t_0$ in $V$ then either $V=V(\theta)$ or the $\Re$-module $V/V(\theta)$ is irreducible.
\end{cor}

\begin{cor}
For any finite-dimensional irreducible $\H$-module $V$, the $\Re$-module $V$ is completely reducible if and only if $t_0$ is diagonalizable on $V$.
\end{cor}

\subsection*{Acknowledgements}

The research is supported by the Ministry of Science and Technology of Taiwan under the project MOST 106-2628-M-008-001-MY4.

\bibliographystyle{amsplain}
\bibliography{MP}

\end{document}